\documentclass[12pt,reqno]{amsart}
\usepackage{graphicx, fullpage,bbm}
\usepackage{amssymb, amsmath, amsthm,tikz, bm,bbm}
\usepackage{epstopdf,enumitem,esint}
\usepackage{verbatim,mathabx}
\usepackage{hyperref}
\newcommand{\be}{\begin{equation} }
\newcommand{\ee}{\end{equation}}
\newcommand{\bse}{\begin{subequations}}
\newcommand{\ese}{\end{subequations}}

\newcommand{\LV}{\left|}
\newcommand{\RV}{\right|}
\newcommand{\LN}{\left\|}
\newcommand{\RN}{\right\|}
\newcommand{\LB}{\left[}
\newcommand{\RB}{\right]}
\newcommand{\LC}{\left(}
\newcommand{\RC}{\right)}
\newcommand{\LA}{\left<}
\newcommand{\RA}{\right>}
\newcommand{\LCB}{\left\{}
\newcommand{\RCB}{\right\}}

\newcommand{\R}{\mathbb{R}} 
\newcommand{\N}{\mathbb{N}} 
\newcommand{\Z}{\mathbb{Z}} 
\newcommand{\T}{\mathbb{T}}

\newcommand{\Div}{\text{div}}

\newcommand{\IE}{\mathcal{E}^0}
\newcommand{\id}{\textrm{I}_n}

\newcommand{\sym}{\mathcal{S}^{n\times n}}

\newcommand{\vn}{V_{\text{new}}}

\newcommand{\ro}{{\rho_0}}
\newcommand{\vo}{{V_0}}
\newcommand{\uo}{{U_0}}
\newcommand{\told}{{R_0}}
\newcommand{\tv}{\widetilde V}
\newcommand{\tu}{\widetilde U}
\newcommand{\eigenmin}{\lambda_{\text{min}}}
\newcommand{\eigenmax}{\lambda_{\text{max}}}
\newcommand{\trace}{\text{tr}}
\newcommand{\Uset}{\mathcal{U}}

\theoremstyle{plain} 
\newtheorem{theorem}{Theorem}[section]

\newtheorem{lemma}{Lemma}[section]
\newtheorem{proposition}{Proposition}[section]
\newtheorem{definition}{Definition}[section]
\theoremstyle{remark} 
\newtheorem{remark}{Remark}[section]

\theoremstyle{definition} 

\numberwithin{equation}{section}

\title{Global ill-posedness for a dense set of initial data to the Isentropic system of gas dynamics}


\author{Robin Ming Chen}
\address{ Department of Mathematics,
      University of Pittsburgh.}
\email{mingchen@pitt.edu}
\author{Alexis F. Vasseur}
\address{Department of Mathematics,
The University of Texas at Austin.}
\email{vasseur@math.utexas.edu}
\author{Cheng Yu}
\address{Department of Mathematics,
University of Florida.}
\email{chengyu@ufl.edu}

\date{}       
\begin{document}

\begin{abstract}
In dimension  $n=2$ and $3$, we show that for any initial datum belonging to  a dense  subset of the energy space, there exist infinitely many global-in-time admissible weak solutions to the isentropic Euler system whenever $1<\gamma\leq 1+\frac2n$. This result can be regarded as a compressible counterpart of the one obtained by Szekelyhidi--Wiedemann (ARMA, 2012) for incompressible flows. Similarly to the incompressible result, the admissibility condition is defined in its integral form. Our result is based on a generalization of a key step of the convex integration procedure. This generalization allows, even in the compressible case, to convex integrate any smooth positive Reynolds stress. A large family of  subsolutions can then be considered. These subsolutions can be generated, for instance,  via regularization of any weak inviscid limit of an associated compressible Navier--Stokes system with degenerate viscosities. 
\end{abstract}

\thispagestyle{empty}
\maketitle

\setcounter{tocdepth}{1}
\tableofcontents

\section{Introduction}\label{sec Intro}

The motion of a compressible fluid in gas dynamics with constant entropy in  the periodic box $\T^n :=[0,1]^n$ for $n = 2$ or $3$ can be modeled by the isentropic Euler system consisting of $n+1$ dynamical equations for the macroscopic state variables: the gas density $\rho = \rho(x,t)$ and the fluid velocity $v = v(x,t)$. The corresponding Cauchy problem reads
\begin{subequations}\label{isen euler system}
\begin{equation}\label{isentropic euler}
\left\{
\begin{split}
& \partial_t \rho + \Div (\rho v) = 0, \\
& \partial_t (\rho v) + \Div (\rho v \otimes v) + \nabla p(\rho) = 0
\end{split}\right.
\end{equation}
with initial condition
\begin{equation}\label{isentropic euler ic}
\rho|_{t=0} = \rho^0, \quad \rho v|_{t=0} = V^0.
\end{equation}
\end{subequations}
 In the isentropic regime, the pressure $p$  is determined by an isentropic equation of state, $p(\rho) = \rho^\gamma$, where  $\gamma$ is the adiabatic constant of the gas. Throughout the paper, we will assume that 
\begin{equation}\label{restriction_gamma}
1<\gamma\leq 1+\frac{2}{n}.
\end{equation}
In particular, it includes the shallow water equation in dimension 2 ($\gamma$=2), and the monoatomic ideal gas in dimension 3 ($\gamma$=5/3).


\subsection{Admissible weak solutions and main result.}
Solutions to \eqref{isentropic euler} carrying $C^1$ regularity away from vacuum are known to uniquely exist  at least locally in time provided that the initial data are sufficiently smooth. On the other hand, it is well-known that such solutions develop singularities (shock waves) in finite time for a generic class of data; see \cite{smoller2012shock,dafermos2005hyperbolic,benzoni2007multi, BSV}. Understanding how solutions can be extended beyond singularities has been a rich field of study. 

Mathematically, to permit the continuation of solutions after the occurrence of singularity, one is required to work with weak solutions, i.e. bounded solutions to \eqref{isentropic euler} in the sense of distribution. To give a more precise definition, it is more convenient to reformulate \eqref{isentropic euler} in terms of conservative variables $(\rho, V)$ where $V := \rho v$.

\begin{definition}[Global weak admissible solutions for compressible Euler equations]\label{def wk soln comp}
We say $(\rho, V)\in L^\infty(\R_+; L^\gamma(\T^n)) \times L^\infty(\R_+; L^{\frac{2\gamma}{\gamma + 1}}(\T^n))$ is a weak solution of \eqref{isentropic euler} on $\R_+$ if
\begin{subequations}\label{weak soln form}
\begin{itemize}
\item $\rho \ge 0$ a.e. and 
\begin{equation}\label{weak mass eqn}
\int^\infty_0 \int_{\T^n} \LC \rho \partial_t \varphi + V \cdot \nabla \varphi \RC \,dx dt = - \int_{\T^n} \rho^0 \varphi(\cdot, 0) \,dx
\end{equation}
for any $\varphi \in C^\infty_c(\T^n \times \R_+)$.
\item $V = 0$ whenever $\rho = 0$ and
\begin{equation}\label{weak momentum eqn}
\int^\infty_0 \int_{\T^n} \LC V \cdot \partial_t \phi + {V \otimes V \over \rho} : \nabla \phi + p(\rho) \Div \phi \RC\,dxdt = -\int_{\T^n} V^0 \cdot \phi(\cdot, 0) \,dx
\end{equation}
for any $\phi \in C^\infty_c(\T^n \times \R_+; \R^n)$, where $V^0 := \rho^0 v^0$.
\item The following global energy inequality holds
\begin{equation}\label{weak energy ineq}
\int_{\T^n} E(\rho, V)(\cdot, t) \,dx \le \int_{\T^n} E(\rho^0,V^0) \,dx \quad \text{for all }\  t \geq 0, 
\end{equation}
where $E(\rho,V) := \frac{\rho^\gamma}{\gamma-1}+  {|V|^2 \over {2\rho}}$ is the total energy. 
\end{itemize}
\end{subequations}
\end{definition}
In the physical variables $(\rho, V=\rho v)$, $E(\rho,V)= \frac{\rho^\gamma}{\gamma-1}+  \rho{|v|^2 \over 2}$.
The integral form of the global energy inequality \eqref{weak energy ineq} is enough to ensure the strong/weak uniqueness result for Lipschitz solutions (see Dafermos \cite{Dafermos} and Di Perna \cite{DiPerna}). The solutions verifying this condition are  called {\it admissible} in the context of convex integration for incompressible flows (see for instance \cite{SW2012ARMA}).  We recall that Equations \eqref{isentropic euler}, together with the a priori bounds from Definition \eqref{def wk soln comp},  imply that $(\rho,V)$ is bounded in $C^0(\R_+, L^q(\T^n)\text{\it-weak})$, for some $q>1$ depending on $\gamma$. Therefore the function $(\rho, V)$ can be defined for all time $t\in \R_+$ (as a function in $L^q(\T^n)$). This justifies the fact that \eqref{weak energy ineq} makes sense for {\it every} time $t\geq0$. Note however that the meaningful constraint of   Inequality \eqref{weak energy ineq} is  that its right hand side   corresponds to the energy of the initial value. The convexity of the energy $E$ in the variables  $(\rho, V)$ implies that if the  inequality
 \eqref{weak energy ineq} holds   for almost every $t>0$, then it holds actually for {\it every} time $t\geq 0$. Let us  now state the main result of this paper.
\begin{theorem}\label{thm main comp}
Assume that $\gamma$ verifies \eqref{restriction_gamma}. 
Then, for any $\varepsilon > 0$ and any $(\varrho^0, U^0)$ such that $E(\varrho^0, U^0) \in L^1(\T^n)$, there exist infinitely many $(\rho^0, V^0)$ satisfying
\begin{equation}\label{comp ic}
\rho^0 > 0, \qquad E(\rho^0, V^0) \in L^1(\T^n), \qquad \|\rho^0 - \varrho^0\|_{L^\gamma(\T^n)}^\gamma + \LN \frac{V^0}{\sqrt{\rho^0}} - \frac{U^0}{\sqrt{\varrho^0}} \RN_{L^2(\T^n)}^2 < \varepsilon,
\end{equation}
such that, for each of such initial values $(\rho^0, V^0)$, there exist  infinitely many global admissible weak solutions $(\rho, V)$ to the compressible Euler equation \eqref{isen euler system} in the sense of Definition \ref{def wk soln comp}.
\end{theorem}
The above theorem provides a dense subset of the energy space, such that any initial value in this set generates infinitely many energy decreasing global weak solution to the isentropic system \eqref{isen euler system} defined on the whole space $\T^n\times \R_+$. This result can be seen as a compressible counterpart of Theorem 2 from 
Sz\'{e}kelyhidi--Wiedemann \cite{SW2012ARMA} which considers incompressible flows (see discussion in the next subsection). It shows that the isentropic system endowed with the global energy criterion is definitively ill-posed for a dense family of initial values. 

\vskip0.3cm
The proof relies on the {\it convex integration} machinery developed by De Lellis--Sz\'ekelyhidi \cite{de2009euler,de2010admissibility}. Although the focus of their work was first on the incompressible Euler equation, a first application to the compressible isentropic Euler was already  present in \cite{de2010admissibility}. For compressible flows, the general strategy  always involves  constructing global density functions such that a convex integration process can be  performed on the  momentum field $V$. The development of the technique for the isentropic case is following   two main directions. One direction, pioneered by Chiodaroli  in \cite{Chiodaroli-continuous-density}, considers a wide class of initial densities. In this situation, the set of initial momentum $V^0$ cannot be chosen a priori, but depends on the convex integration procedure. The original result \cite{Chiodaroli-continuous-density} treats general $C^1$ initial densities, and was later extended to the case of possibly discontinuous piecewise $C^1$ functions by Luo--Xie--Xin \cite{Xin},
and Feireisl \cite{Feireisl-early}. The other direction, pioneered by Chiodaroli--De Lellis--Kreml \cite{chiodaroli2015global}, focuses on initial values being Riemann data.  They are piece-wise constant functions with a unique planar set of discontinuities. The situation of a shock was first considered, and later  extended to other Riemann problems  (see \cite{brezina2018contact}, \cite{Chiodaroli18}). Extensions of both strategies have been studied for the full Euler system (see for instance  Chiodaroli--Feireisl--Kreml \cite{Full1} or Al Baba--Klingenberg--Kreml--M\'{a}cha--Markfelder \cite{Full2}). A natural problem consists in  studying  the size of  the class of initial values leading to non-unique solutions. Note that the energy condition  \eqref{weak energy ineq} is crucial. Without this admissibility condition, non-unique solutions to \eqref{isen euler system} can be constructed for any fixed initial values (see Abbatiello--Feireisl \cite{Abbatiello}). 

\subsection{The incompressible case.}

Let us now consider an  incompressible ideal flow with density being normalized to unity, whose dynamics is governed by the incompressible Euler equations
\begin{subequations}\label{incomp euler system}
\begin{equation}\label{incomp euler}
\left\{
\begin{split}
& \partial_t v + \Div (v \otimes v) + \nabla p = 0, \\
& \Div \,v = 0, \\
\end{split}\right.
\end{equation}
with initial datum
\begin{equation}\label{incomp euler ic}
v|_{t=0} = v^0,
\end{equation}
\end{subequations}
where now the pressure $p$ arises as a Lagrange multiplier due to the incompressibility condition. For an initial velocity field $v^0 \in L^2(\T^n)$ with $\Div \, v^0 = 0$, the corresponding notion of global  in time admissible weak solutions  is given as follows.
\begin{definition}[Global weak solutions for admissible  incompressible Euler equations]\label{def wk soln incomp}
We say $v \in L^{\infty}(\R_+;L^2(\T^n)) $ is a weak solution of \eqref{incomp euler system} if it is divergence-free in the sense of distribution and
\begin{subequations}\label{weak soln form incomp}
\begin{itemize}
\item for any $\phi \in C^\infty_c(\T^n \times \R_+; \R^n)$ with $\Div \, \phi = 0$,
\begin{equation}\label{weak incomp euler eqn}
\int^\infty_0 \int_{\T^n} \LC v \cdot \partial_t \phi + v \otimes v : \nabla \phi \RC \,dxdt =  -\int_{\T^n} v^0 \cdot \phi(\cdot, 0) \,dx.
\end{equation}
\item The following global energy inequality holds
\begin{equation}\label{weak energy ineq incomp}
\int_{\T^n} \frac12 |v(\cdot, t)|^2 \,dx \le \int_{\T^n} \frac12 |v^0(\cdot)|^2 \,dx \quad \text{for every  }\  t\geq 0. 
\end{equation}
\end{itemize}
\end{subequations}
\end{definition}
Similarly to the compressible case, any such solution actually lies in $C^0(\R_+; L^2(\T^n)\text{\it-weak})$, and so \eqref{weak energy ineq incomp} can be written for {\it every} time $t\geq 0$. However, still because of the convexity of the energy, it is enough to check that Inequality \eqref{weak energy ineq incomp} is true for almost every  $t>0$. We now state our result in the incompressible case.
\begin{theorem}\label{thm main incomp}
For any $\varepsilon > 0$ and any $u^0 \in L^2(\T^n)$, there exist infinitely many $v^0 \in L^2(\T^n)$ satisfying
\begin{equation}\label{incomp ic}
\|v^0 - u^0\|_{L^2(\T^n)}^2 < \varepsilon,
\end{equation}
such that for each such initial value $v^0$, there exist  infinitely many global weak solutions $v$ 
to the incompressible Euler equation \eqref{incomp euler system} in the sense of Definition \ref{def wk soln incomp}.
\end{theorem}

We want to remark that the above result is not new in the context of incompressible Euler equations. It  was first proved in \cite{SW2012ARMA} by Sz\'{e}kelyhidi--Wiedemann, and was later improved with the construction of $C^{1/5}$ solutions in Daneri--Runa--Sz\'{e}kelyhidi \cite{uptoCont}, and to $C^{\alpha}$ solutions up to the Onsager range $\alpha<1/3$ in Daneri--Sz\'{e}kelyhidi \cite{uptoOnsager}.
We will nevertheless give a proof of Theorem \ref{thm main incomp} which unifies the compressible and incompressible points of view in the context of the $L^\infty$ theory.

\subsection{Main ideas of the proof.} So far, all constructions of non-unique solutions   for compressible flows are done with the  $L^\infty$ theory of convex integration.  The general strategy follows  two  steps: the construction of  {\it subsolutions}, and the convex integration of these subsolutions to obtain actual solutions (see for instance  \cite{de2010admissibility}).  In their more general form, subsolutions are functions $(\rho, V=\rho u, R)$ solving the so-called ``Euler--Reynolds'' system
\begin{equation}\label{eqn subsoln}
\left\{
\begin{split}
\partial_t \rho + \Div V & = 0,\\
\partial_t V + \Div \left(\frac{V \otimes V}{\rho} + p(\rho)\id +  R\right) & = 0, 
\end{split}\right.
\end{equation}
where the compressible ``Reynolds stress tensor''  $R(t,x)$ is a positive semi-definite symmetric matrix for every $x,t$.  The family of subsolutions is stable under weak limit, or convex combination, therefore it is far easier to construct subsolutions than solutions (which corresponds to $R=0$). The convex integration provides a way to construct infinitely solutions to the Euler equations, from a subsolution, for a certain family of Reynolds stresses $R$. The more general the family of Reynolds stresses processable via the convex integration, the easier it is to construct subsolutions, and the larger is  the set of initial values which can be reached. To the best of the authors' knowledge, in the context of compressible fluids, the convex integration technique used so far allows to deal with only {\it diagonal} Reynold stresses (see \cite{Chiodaroli-continuous-density, chiodaroli2015global}). Such a method  is a variant from the incompressible case \cite{de2009euler}. It states that for every open set $P$, and every $\rho, q$ positive real-valued functions, $V$ vector-valued function, and $U$ traceless symmetric matrix-valued function  (all smooth enough), through convex integration there exist infinitely  many $\widetilde V$ and  traceless $\widetilde U$ (as oscillatory perturbations), both compactly supported  in $P$, such that in $\R^n\times \R_+$:
\begin{equation}  \label{eq0}
\left\{ \begin{split}
\Div   \widetilde V & = 0, \\
\partial_t   \widetilde V+ \Div \widetilde U & =0, 
\end{split} \right. 
\end{equation}
while in $P$ a nonlinear constraint 
\begin{equation} \label{eq1}
\frac{(V+\widetilde V)\otimes (V+\widetilde V)}{\rho}-(U+\widetilde U)=\left(\frac{|V|^2}{n\rho}+q\right) \id
\end{equation}
is achieved as to eliminate the Reynolds stress $R :=q\id $. Therefore under the assumption that there exists a smooth enough, energy-compatible  subsolution $(\rho, V, R)$ of \eqref{eqn subsoln} and denoting $U :=(V\otimes V-\id |V|^2/n)/\rho$, the oscillatory perturbations $(\widetilde V, \widetilde U)$ constructed from \eqref{eq0}--\eqref{eq1} readily generate $(\rho,V+\widetilde V)$ as solutions to the the isentropic Euler system.  However, the requirement on the Reynolds stress that $R = q \id$ is stringent and prevents one from generating a large class of initial values.
\vskip0.3cm
\subsubsection*{Convex integration with general Reynolds stresses} One of the main contributions of this paper is the generalization of the key convex integration tool to accommodate {\it any} positive definite Reynolds stresses (see Lemma \ref{lem conv int}) in the $L^\infty$ framework. Namely, we show that we can construct infinitely many solutions of \eqref{eq0}, replacing the contraint \eqref{eq1} with 
\begin{equation} \label{eq2}
\frac{(V+\widetilde V)\otimes (V+\widetilde V)}{\rho}-(U+\widetilde U)=\frac{|V|^2}{n\rho} \id+R,
\end{equation}
for any continuous strictly positive Reynolds stress $R > 0$. We state and prove this result in both $\R^n$ and $\T^n$ for future use. 
\vskip0.3cm
As in the previous work, this result is obtained by partitioning the domain $P$ in small areas where $\rho, V$ and $R$ are almost constant, and so considering the generation of highly oscillatory perturbations for the constant case first. Denote $ \mathcal{S}_0^n$ the set of traceless symmetric matrices in dimension $n$.  In the previous case when $R = q \id$, the generation of oscillations is based on the study of  the convex set:
\begin{equation}\label{eq3}
K_{d,r}^{co} :=\left\{(V,U)\in \R^n\times \mathcal{S}_0^n: e_d(V,U)\leq \frac{r^2}{2} \right\},
\end{equation}
where $e_d(V,U) := (n/2)\lambda_{\rm max}(V\otimes V-U)$, with $\lambda_{\rm max} (w)$ denoting the largest eigenvalue of the matrix $w$ (see \cite{de2010admissibility}). 
The oscillatory perturbations $(\widetilde V, \widetilde U)$ have to be constructed such that  for all time and space, $(V+\widetilde V, U+\widetilde U)$ stay in the set $K_{d,r}^{co}$ defined with 
$r^2=|V|^2+\trace R$.
The first observation is that oscillatory perturbations for a constant (but possibly non-diagonal) Reynolds stress $R=\mathring R+ \id (\trace R/n)$  can be constructed similarly as in \cite{de2010admissibility}, using the ``translation'' of $K_{d,r}^{co}$:
$$
K_r^{co} := (0,\mathring R)+K_{d,r}^{co}.
$$
This can be done in an admissible way as long as $\lambda_{\rm min}(R)$, the smallest eigenvalue of $R$, is positive. 
\vskip0.3cm
The difficulty is then to integrate this building block through the general convex integration scheme. In the classical situation, the problem \eqref{eq0}--\eqref {eq1}
is replaced by a relaxed one where \eqref{eq1} is replaced by a (matrix) inequality. That is, by the property that there exists a positive semidefinite matrix-valued function $S$ such that for all $(x,t)\in P$:
\begin{equation}\label{eq1bis}
\frac{(V+\widetilde V)\otimes (V+\widetilde V)}{\rho}-(U+\widetilde U)+S=\left(\frac{|V|^2}{n\rho}+q\right) \id. 
\end{equation}
The general convex integration procedure (see \cite{de2009euler}) ensures, via a topological 
Bair\'e category argument, the existence of infinitely many solutions to the relaxed problem \eqref{eq0} and \eqref{eq1bis} with the following property: each one of these solutions  cannot be reached  via a sequence of oscillatory  solutions to the same relaxed problem (namely, it is not possible to find a sequence of solutions to  \eqref{eq0} and \eqref {eq1bis} which converges weakly to this special solution, while not converging strongly). For incompressible flows \cite{de2009euler,de2010admissibility}, or ``piece-wise incompressible'' flows (compressible flows with a piece-wise constant in space and time-independent density) \cite{de2010admissibility,chiodaroli2015global}, or ``semi-stationary'' flows (time-independent density) \cite{Chiodaroli-continuous-density,aw2021sima}, constraint \eqref{eq1bis} can be designed in such a way that $S$ takes the form of a multiple of the identity matrix. This particularly allows one to derive, on those solutions, a ``saturation'' property of $S$ in the sense that $\lambda_{\rm min}(S)=0$. Taking advantage of the form that $S$ takes, this further concludes that $S \equiv 0$, meaning that those infinitely many functions  are actually solutions to \eqref{eq0}--\eqref{eq1}. 
\vskip0.3cm
On the other hand, the challenge in extending the framework of \cite{de2009euler} to \eqref{eq0} and \eqref{eq2} is apparent: when $q\id$ in \eqref{eq1bis} is replaced by a general positive matrix $R$, the corresponding $S$ is generically non-diagonal. A na\"ive adaptation of the convex integration as indicated above would still lead to a saturation in terms of $\lambda_{\rm min}(S)=0$. However this is never strong enough to imply the vanishing of $S$ any more. The resolution we propose here is to exploit the additional saturation in $\lambda_{\rm max}(S)$ in the course of the convex integration. In particular, when $\lambda_{\rm min}(S)>0$ on $P$, we will construct oscillatory perturbations with oscillation strength proportional to $\int_P \trace S(x,t)\,dx\,dt$. This way, we will verify that the solutions selected by the Bair\'e argument verify both $\lambda_{\rm min}(S) = 0$ on $P$ and $\int_P \trace S(x,t)\,dx\,dt=0$.
 The condition on  $\lambda_{\rm min}(S)$ indicates that all the eigenvalues of $S$ are nonnegative, and from the condition on the $\trace S$, their sum is 0 almost everywhere. This implies that $S=0$ on $P$ and so these subsolutions are actually solutions. 
 \vskip0.3cm
\subsubsection*{Density of wild initial data and double convex integration}
When considering the Cauchy problem \eqref{isen euler system}, the initial data that can lead to infinitely many admissible weak solutions are termed the ``wild'' initial data \cite{de2010admissibility}. In the context of incompressible flows, it has been shown that wild initial data are $L^2$-dense for $L^\infty$ weak solutions \cite{SW2012ARMA} as well as for H\"older $C^\alpha$ weak solutions \cite{uptoCont,uptoOnsager}. One of the subtleties in dealing with the Cauchy problem is that the subsolutions need to be adjusted to capture the full initial energy, and that the superimposed oscillatory perturbations need to preserve the initial datum. This is achieved by the so-called ``double convex integration'' first introduced in \cite{de2010admissibility} for $L^\infty$ solutions, and later extended to treat H\"older solutions \cite{daneri2014,uptoCont,uptoOnsager}. Specifically, a time-localized convex integration is first performed to construct a nontrivial subsolution with its wild initial datum, followed by a second convex integration to pass from this subsolution to infinitely many weak solutions. As is pointed out in \cite{uptoOnsager}, such a strategy is required in proving the density of the wild initial data.
\vskip0.3cm
One of the key ingredients of the above strategy is to find an appropriate class of perturbations in the scheme capable of generating sufficiently rich family of positive definite Reynolds stresses, from which a suitable notion of subsolutions can be introduced to track the relation between the size of the Reynolds stress and the loss of regularity. In the $C^\alpha$ theory, a fairly precise control of the H\"older norms at each iteration step is needed. In particular, the full strength of the Reynolds stresses is used in the estimates. Mikado flows are thus used to allow any positive definite Reynolds stresses throughout the iteration, since the Beltrami flows are not sufficient \cite{Choffrut}. In contrast, the notion of subsolutions in the $L^\infty$ framework is much less rigid and the solutions can be obtained implicitly via the Bair\'e argument. Only a portion of the size of the Reynolds stresses is needed in the estimate and hence Beltrami flows suffice the role of fast oscillating perturbations.
\vskip0.3cm
For compressible flows with a varying density, on the other hand, as explained in the earlier context of this subsection, the Reynolds stress $R$ in the $L^\infty$ scheme takes a general form while the oscillations need to have strength proportional to the size of $R$ measured through its trace as $\int_P \trace R(x,t)\,dx\,dt$. We want to emphasize that it is essential to allow a general class of $R$ in the convex integration in order to cover a large family of initial values. 
\vskip0.3cm
To ensure a full saturation of the initial energy for the subsolutions, we follow a similar version of the double convex integration on a small interval $[0,T]$ first, and then, on $[T, \infty)$. The weak solutions directly constructed by convex integration may not verify \eqref{weak energy ineq}. But for each fixed one, its time shifts, $u_s(t)=u(s+t)$ will verify it for almost every $s>0$. Taking $s$ small enough, we can show that this provides infinitely many initial value with (at least) one admissible solution on an interval $[0,T-s]$. Considering the same time shift on the functions obtained through convex integration on $[T,\infty)$ provides infinitely many continuation on $[T-s,\infty)$ combined with each admissible solutions first constructed on $[0,T-s]$. Note that by construction, the solutions are continuous in time at $T-s$ (weakly in $x$) and their value is exactly the value of the subsolution at this time. See Figure \ref{fig convint}. 
\vskip0.3cm
The above method works very effectively on incompressible flows. However additional care is needed in the compressible case. The total energy consists of both the kinetic and potential parts. Moreover, the Reynolds stress, which can be thought of as a result of commuting weak limits with nonlinearity of the Euler equations, also involves information about fluctuation in both velocity (or momentum) and density components. Since our convex integration is designed such that the ``defect energy'' of the subsolutions is injected into the kinetic energy, it is possible that there is a loss of the total energy resulting from the potential energy. Therefore before the first convex integration, some compensating potential energy should be pumped into the Euler--Reynolds system. Such an energy requirement imposes the constraint on the adiabatic exponent $\gamma \le 1 + \frac2n$ (see below for more detailed explanation).
\vskip0.3cm
\subsubsection*{Construction of energy-compatible subsolutions} Now that we are able to convex integrate with any smooth positive Reynold stresses, the construction of subsolutions is highly simplified. We choose to construct them from the weak inviscid limit of Navier--Stokes equations. For fixed viscosities $\nu$, the standard existence theory requires $\gamma>3/2$ (see Feireisl--Novotny--Petzeltova \cite{F32}). For this reason, we are  using instead a Navier--Stokes model with degenerate viscosities constructed in \cite{VY2016, LV2018,bresch2019global} which allows $\gamma>1$. We then modify the inviscid limit obtained from this model to ensure that the density $\rho$ and the Reynolds stress $R$ are smooth enough, and $\lambda_{\rm min}(R) > 0$ globally. Note that for compressible flows $R$ consists of two parts $\mathcal R$ and $r \id$ arising from the averaging effect on the velocity and on the density through the pressure, respectively. Such a weak inviscid limit (together with the smoothing process) results in an energy density
\begin{equation*}
\widetilde e := \frac12 \LC \frac{|V|^2}{\rho} + \trace \mathcal R \RC + \frac{p(\rho) + r}{\gamma - 1},
\end{equation*}
where both the kinetic and potential energies are changed. On the other hand, our convex integration produces subsolutions having energy density
\begin{equation*}
\bar e := \frac12 \LC \frac{|V|^2}{\rho} + \trace R \RC + \frac{p(\rho)}{\gamma - 1} = \frac12 \LC \frac{|V|^2}{\rho} + \trace \mathcal R + n r \RC + \frac{p(\rho)}{\gamma - 1},
\end{equation*}
from which one sees that the entire defect energy is injected into the kinetic energy through the convex integration (since the oscillations are imposed on velocity only). Clearly we need $\bar e \le \widetilde e$, which results in \eqref{restriction_gamma}.
From \eqref{weak energy ineq} we see that the energy compatibility requirement corresponds to asking $\int_{\T^n} \widetilde e \,dx \le \int_{\T^n} E(\rho^0,V^0) \,dx$. Therefore the construction of the energy-compatible subsolutions involves careful adjustments on $R$ through the regularization, positivity enhancement, and energy compatibility procedures. We are able to show that these adjustments can be done in a unified way using an abstract lemma about convex combination of subsolutions, cf. Lemma \ref{lem conv combo comp} (and Lemma \ref{thm ci incomp} for the incompressible case).
\vskip0.3cm
The rest of the paper is as follows. Section \ref{sec bb conv} is dedicated to the convex integration of \eqref{eq0} \eqref{eq2} in the case where $\rho, V, U$ and $R$ are constants.  The general case is treated in  Section \ref{sec ci}. Section \ref{sec incomp} is dedicated to the proof of Theorem \ref{thm main incomp} for the incompressible case, and Section \ref{sec comp} to the proof of Theorem \ref{thm main comp} for the compressible case.

\section{Building blocks for convex integration}\label{sec bb conv}

Recall from the Introduction that our focus is to consider solutions to \eqref{eqn subsoln} with $R > 0$ being positive definite in the interior region as the `subsolution' to the isentropic Euler system \eqref{isentropic euler}. 
Note that equation \eqref{eqn subsoln} is equivalent to 
\begin{equation}\label{isentropic equivalent}
\left\{
\begin{split}
\partial_t \rho + \Div V & = 0, \\
\partial_t V + \Div U + \nabla \left( p(\rho) + {|V|^2 \over n\rho}  \right) + \Div  R & = 0,
\end{split}\right.
\end{equation}
where $V := \rho v$ and $\displaystyle U := {V\otimes V \over \rho} - {|V|^2 \over n\rho} \id$.

Consider a $C^0$ solution $(\ro, \vo, \told)$ to \eqref{isentropic equivalent}. Denote the bounds for $\ro$ to be
\begin{equation}\label{bounds rho}
0 < {1\over \Lambda^2} \le \ro \le \Lambda^2.
\end{equation}
The goal is to construct infinitely many bounded solutions $(\widetilde V, \widetilde U)$ supported in a given domain $P$ satisfying
\begin{subequations}\label{ci system}
\begin{equation}  \label{ci eqn} 
\left\{ \begin{split}
\Div \widetilde V & = 0, \\
\partial_t \widetilde V + \Div \widetilde U & = 0, 
\end{split} \right. 
\end{equation}
with
\begin{equation} \label{ci constraint =}
{(\vo + \widetilde V) \otimes (\vo + \widetilde V) \over \ro} - (\uo + \widetilde U) =  {|\vo|^2 \over n\ro} \id + \told \quad \text{a.e. } P,
\end{equation}
where 
\begin{equation}\label{def U_0}
U_0 := {V_0\otimes V_0 \over \ro} - {|V_0|^2 \over n\ro} \id.
\end{equation}
\end{subequations}
The construction of $(\widetilde V, \widetilde U)$ will be addressed in the next section. As a building block, we will start with a simplified setting described below. 

\subsection{Constant states problem}
we will first consider a simplified problem of \eqref{ci system}, namely when $V_0$, $\ro$ and $R_0$ are constant vector, constant scalar and constant symmetric matrix respectively and satisfy $\ro > 0, \ R_0 > 0$. Therefore $U_0 \in \sym_0$ is also a constant matrix. Apparently such a $(\rho_0, V_0, R_0)$ solves \eqref{isentropic equivalent}.

Introducing $(V, U)(x,t) := \LC \tv/\sqrt{\rho_0}, \tu\RC(x, t\sqrt{\rho_0})$, then  $(V, U)$ satisfies
\begin{subequations}\label{ci scaled}
\begin{equation}\label{ci simplified}
\left\{ \begin{split}
\Div V & = 0, \\
\partial_t V + \Div U & = 0, 
\end{split} \right.
\end{equation}
with
\begin{equation}\label{simplified ci =}
(V_0 + V) \otimes (V_0 + V) - (U_0 + U) = {C_0 \over n} \id + R_0
\end{equation}
where $C_0 := {|\vo|^2 \over \ro} > 0$,  $\told > 0$ is positive definite, and $V_0$ is relabeled as $V_0/\sqrt{\rho_0}$.
\end{subequations}

Following \cite{de2010admissibility}, for $r\ge 0$ we define the states of speed $r$
\begin{equation*}
K_r := \LCB (V, U) \in \R^n \times \sym_0:\ U = V \otimes V - R_0 - {1\over n}\LC r^2 - \trace R_0 \RC \id, \ |V| = r \RCB.
\end{equation*}
Denote $K^{co}_r$ the convex hull in $\R^n \times \sym$ of $K_r$. Also define
\begin{equation*}
e(V, U) := {n\over 2} \eigenmax\LC V \otimes V - U - R_0 \RC
\end{equation*}
where $\eigenmax$ denotes the largest eigenvalue. Then similar to \cite[Lemma 3]{de2010admissibility}, we have the following
%
\begin{lemma}\label{lem convex hull}
For $(V, U) \in \R^n \times \sym_0$ it holds that
\begin{enumerate}[label={\upshape(\roman*)}]
\item $e: \R^n \times \sym_0 \to \R$ is convex;
\item $\displaystyle {1\over 2}\LC |V|^2 - \emph{tr} R_0 \RC \le e(V, U)$, with equality if and only if 
\[
U = V \otimes V - R_0 - {1\over n}\LC |V|^2 - \emph{tr} R_0 \RC \id;
\]
\item denote $|U|_\infty$ the operator norm of $U$, then
\[
|U|_\infty \le {2(n-1) \over n} e(V,U) + (n-1) |R_0|_\infty;
\]
\item the convex hull of $K_r$ is
\begin{equation*}
K^{co}_r = \LCB (V, U) \in \R^n \times \sym_0 : \ e(V, U) \le {1\over 2} (r^2 - \emph{tr} R_0) \RCB;
\end{equation*}
\item for $(v, u) \in \R^n \times \sym_0$, $\sqrt{2\LB e(v, u) + \emph{tr} R_0 \RB}$ gives the smallest $s$ for which $(v, u) \in K^{co}_s$.
\end{enumerate}
\end{lemma}
%
\begin{proof}
The proofs of (i), (ii) and (v) follows almost identically as in \cite[Lemma 3]{de2010admissibility}. So let's focus only on (iii) and (iv).

(iii) Let $\xi_0$ be a unit eigenvector of $U$ associated to its smallest eigenvalue $\eigenmin(U)$.  We have by definition that
\begin{align*}
e(V, U) & \ge {n\over 2} \max_{\xi \in \mathbb S^{n-1}} \big( - \LA \xi, (U + R_0) \xi \RA \big) \\
& \ge {n\over 2} \big( - \LA \xi_0, U \xi_0 \RA \big) - {n\over 2}  \LA \xi_0, R_0 \xi_0 \RA  \\
& \ge -{n\over 2} \eigenmin(U) - {n \over 2} \max_{\xi \in \mathbb S^{n-1}} \big( \LA \xi, R_0 \xi \RA \big) \\
& \ge -{n\over 2} \eigenmin(U) - {n \over 2} \eigenmax(R_0).
\end{align*}
Thus since $U$ is trace free, we have
\[
|U|_\infty \le (n-1) \big( -\eigenmin(U) \big) \le {2(n-1) \over n} e(V,U) + (n-1) |R_0|_\infty.
\]

(iv) Denote
\[
S_r := \LCB (V, U) \in \R^n \times \sym_0 : \ e(V, U) \le {1\over 2} (r^2 - \trace R_0) \RCB.
\]
From definition we see that whenever $(V, U) \in K_r$ we have $e(V, U) = {1\over 2} (r^2 - \trace R_0)$. Since $e$ is convex from (i), it follows that
\[
K^{co}_r \subset S_r.
\]
From (ii) and (iii) we know that $S_r$ is compact. hence $S_r$ equals the closed convex hull of its extreme points. 

From $(V, U) \in S_r \backslash K_r$, we can without loss of generality assume that $V\otimes V - U - R_0$ is diagonal with diagonal entries $\lambda_1 \ge \ldots \ge \lambda_n$ satisfying $\lambda_1 \le {1\over n} (r^2 - \trace R_0)$. From (ii) and the fact that $(V, U) \not\in K_r$ we conclude that $\lambda_1 < {1\over n} (r^2 - \trace R_0)$.

Now we can continuously perturb such $(V, U)$ in $\R^n \times \sym_0$: write $V = \sum_i V^i e_i$ where $e_1, \ldots, e_n$ are the basis vectors. Pick a fixed pair $(v, u) \in \R^n \times \sym_0$ as
\[
v = e_n, \quad u = \sum^{n-1}_{i=1} V^i (e_i \otimes e_n + e_n \otimes e_i).
\]
This way
\[
(V + tv) \otimes (V + tv) - (U + t u) = (V \otimes V - U) + (2t V^n + t^2) e_n \otimes e_n,
\]
and therefore for $|t|$ sufficiently small $e(V+tv, U + tu) \le {1\over 2} (r^2 - \trace R_0)$. Hence $(V+tv, U + tu) \in S_r$, and thus $(V, U)$ is not an extreme point of $S_r$. So all of the extreme points of $S_r$ are contained in $K_r$. 
\end{proof}

\subsection{Oscillations}
The construction of the needed oscillations in the interior of $K^{co}_r$ is done via seeking suitable plane-wave solutions. They correspond to the following admissible segments; see \cite[Definition 6]{de2010admissibility}.
\begin{definition}\label{def adm seg}
Given $r>0$, we call a line segment $\sigma \subset \R^n \times \sym_0$ an {\it admissible segment} if it satisfies
\begin{enumerate}[label=\rm(\alph*)]
\item $\sigma \subset \text{int }K^{co}_r$,
\item $\sigma$ is parallel to $(a, a\otimes a) - (b, b\otimes b)$ for some $a, b \in \R^n$ with $|a| = |b| = r$ and $b \ne \pm a$.
\end{enumerate}
\end{definition}

Similar to \cite[Lemma 4.3]{de2009euler} and \cite[Lemma 6]{de2010admissibility}, we can first record the following geometric property of $K^{co}_r$ which provides the existence of sufficient large admissible segments.
%
\begin{lemma}[Existence of large admissible segements]\label{lem geometric}
Set $N_0 := \text{dim} (\R^n \times \sym_0) = \tfrac{n(n+3)}{2} - 1$. 
For any $r>0$ and for any $(V, U) \in \emph{int }K^{co}_r$ there exists an admissible line segment 
\begin{equation}\label{def line}
\sigma := \Big[ (V, U) - (v, u), (V,U) + (v,u) \Big]
\end{equation}
such that
\begin{equation*}
|v| \ge {1 \over 4N_0 r} \left( r^2 - |V|^2 \right) \quad \text{ and } \quad \emph{dist}(\sigma, \partial K^{co}_r) \ge {1\over2} \emph{dist}((V, U), \partial K^{co}_r).
\end{equation*}
\end{lemma}
%
The proof of this lemma follows directly  from \cite[Lemma 4.3]{de2009euler}  applied on the translated set $K^{co}_r-(0,(1/n)(\trace R_0)\id -R_0)$.\vskip0.1cm

We now recall  \cite[Proposition 4.1]{chiodaroli2015global} which provides  the existence of localized plane waves oscillating between two states of \eqref{ci simplified} with equal speed.
%
\begin{lemma}[Localized plane waves]\label{lem plane waves}
Let $a, b \in \R^n$ such that $a \ne \pm b$ and $|a| = |b|$. For a $\lambda > 0$ consider a segment $\sigma = [-p, p] \subset \R^n \times \sym_0$ where $p = \lambda \LB (a, a\otimes a) - (b, b\otimes b) \RB$. Then there exists a pair $(v,u) \in C^\infty_c(B_1(0) \times (-1,1))$ solving
\begin{equation}\label{osci eqn}
\left\{ \begin{split}
\textup{div}_x v & = 0, \\
\partial_t v + \textup{div}_x u & = 0, 
\end{split} \right.
\end{equation}
and such that
\begin{enumerate}[label={\upshape(\roman*)}]
\item the image of $(v, u)$ is contained in an $\epsilon$-neighborhood of $\sigma$ and $\int (v, u) \,dxdt = 0$;
\item $\int |v(t,x)|\,dxdt \ge \alpha \lambda |b - a|$ where $\alpha > 0$ is a geometric constant.
\end{enumerate}
\end{lemma}
%

\subsection{Perturbation property}
In this subsection we will derive a key property  which will be used in Section \ref{sec ci}.

Let $C_0 \geq 0$ be a constant and $R_0 > 0$ be a symmetric positive definite matrix. Define a subset of $\R^n \times \sym_0$
\begin{equation*}
\Uset := \LCB (v,u)\in \R^n \times \sym_0:\ v\otimes v - u - R_0 < {C_0 \over n} \id \RCB.
\end{equation*}
Also define a function space
\begin{equation}\label{X_0 const}
\begin{split}
X_0^c := & \Big\{ (V, U) \in C^\infty_c(P; \R^n \times \sym_0): \ (V, U) \text{ solves } \eqref{ci simplified} \text{ and }  \\
&\qquad \qquad \left. (V_0 + V) \otimes (V_0 + V) - (U_0 + U) < {C_0 \over n} \id + R_0 \RCB.
\end{split} 
\end{equation}

Recasting Lemma \ref{lem geometric} on $\Uset$ we have
%
\begin{lemma}[Geometric property of $\Uset$]\label{lem geometric U}
There exists a positive geometric constant $c_0$ such that for any $(\tilde v, \tilde u) \in \Uset$, there exists a segment $\sigma$ as in Lemma \ref{lem plane waves} with $|a| = |b| = \sqrt{C_0 + \emph{\trace} R_0}$, 
\begin{equation*}
(\tilde v, \tilde u) + \sigma \in \Uset, \quad \text{and} \quad \lambda |b - a| \ge c_0\LC C_0 + \emph{\trace} R_0 - |\tilde v|^2 \RC.
\end{equation*} 
\end{lemma}
%
\begin{proof}
From Lemma \ref{lem convex hull} we see that 
\[
\Uset = \text{int }K^{co}_r, \quad \text{where }\quad r^2 = C_0 + \trace R_0.
\]
The existence of the claimed segment $\sigma$ is a direct consequence of Lemma \ref{lem geometric}. Moreover since the length of $\sigma$ is (up to a geometric constant, say, $c_0$) comparable to $\lambda |b - a|$, the conclusion of the lemma holds.
\end{proof}

Now we can conclude this section with the following $L^1$-coercivity result.
%
\begin{proposition}[$L^1$-coercivity of the perturbation]\label{prop L^1 coercive}
There exists a  constant $c_1 > 0$ such that the following is true.
Let $(V,U) \in X_0^c$ where $X_0^c$ is defined in \eqref{X_0 const}.  Then, for any  open set $\Gamma \subset P$, there exists a sequence $\{(V_i, U_i)\} \subset X_0^c$ converging weak-$\ast$ to $(V, U)$ such that 
\begin{equation}\label{L^1 coercive constant}
\LN V_i - V \RN_{L^1(\Gamma)} \ge c_1 \LB \LC C_0 + \emph{tr} R_0 \RC |\Gamma| - \LN V_0 + V \RN^2_{L^2(\Gamma)} \RB.
\end{equation}
\end{proposition}
%

\begin{proof}
Fix any point $(x_0, t_0) \in \Gamma$ and note that $(V, U) + (V_0, U_0)$ takes values in $\Uset$. Applying Lemma \ref{lem geometric U} yields the segment $\sigma$ with $(\tilde v, \tilde u) = (V(x_0, t_0), U(x_0, t_0)) + (V_0, U_0)$. Choose $r>0$ such that $(V(x, t), U(x, t)) + (V_0, U_0) + \sigma \subset \Uset$ for any $(x, t) \in B_r(x_0) \times (t_0 - r, t_0 + r)$.  It exists thanks to the continuity of $(V,U)$.

For any $\epsilon > 0$ consider a pair $(v, u)$ as in Lemma \ref{lem plane waves} and define 
\[
(v_{x_0,t_0,r}, u_{x_0,t_0,r})(x, t) := (v, u) \LC \tfrac{x - x_0}{r}, \tfrac{t - t_0}{r} \RC.
\]
Clearly, for $\epsilon$ small enough,  $(V, U) + (v_{0,r}, u_{0,r}) \in X_0^c$. Moreover
\begin{equation}\label{L1 est1}
\int_{B_r(x_0) \times (t_0 - r, t_0 + r)} |v_{x_0,t_0,r}| \,dxdt \ge \alpha c_0 \LC C_0 + \trace R_0 - |V_0 + V(x_0, t_0)|^2 \RC r^{n+1}.
\end{equation}
By continuity there exists an $r_0$ such that for all $r< r_0$ the above holds for every $(x,t)$ with $B_r(x) \times (t - r, t + r) \subset \Gamma$. 

Set $r = \tfrac{1}{k} < r_0$ and pick finitely many points $(x_j, t_j)$ such that $B_r(x_j) \times (t_j - r, t_j + r) \subset \Gamma$ are pairwise disjoint and satisfy
\begin{equation}\label{L1 est2}
\sum_j \LC C_0 + \trace R_0 - |V_0 + V(x_0, t_0)|^2 \RC r^{n+1} \ge \bar c \LC \LC C_0 + \trace R_0 \RC |\Gamma| - \int_\Gamma |V_0 + V(x,t)|^2\,dxdt \RC
\end{equation}
for some geometric constant $\bar c > 0$. So now we define
\[
(V_k, U_k) := (V, U) + \sum_j (v_{x_j,t_j,r}, u_{x_j,t_j,r}).
\]

It is clear that $(V_k, U_k) \in X_0^c$ since the supports of $(v_{x_j,t_j,r}, u_{x_j,t_j,r})$ are pairwise disjoint. Moreover $(V_k, U_k) \rightharpoonup^* (V, U)$ in $L^\infty$. Finally we see that \eqref{L^1 coercive constant} follows from the above two estimates \eqref{L1 est1} and \eqref{L1 est2}. 
\end{proof}

\begin{remark}
Depending on the values of $V_0, U_0, R_0$, the set $X_0^c$ may be empty. In this case Proposition \ref{prop L^1 coercive} is void, but still holds true.
\end{remark}
\begin{remark}\label{rem size}
Taking the trace of element $X_0^c$ shows that:
$$
\sup_{P}|V_i|\leq 2 |V_0|+2(C_0+\trace  R_0).
$$
\end{remark}

\section{Discretization and convex integration}\label{sec ci}

Now let's come back to the system \eqref{ci system}, but with $\ro$, $R_0$ being possibly non-constant functions. At this point we do not restrict ourselves to only consider $(\ro, \vo, U_0, R_0)$ to be a solution to \eqref{isentropic equivalent}, but to be some general continuous functions such that on an open set $P$, \eqref{def U_0} holds true, $R_0>0$ as a matrix, and $\rho_0$ verifies a uniform condition as \eqref{bounds rho}. The goal is to construct infinitely many solutions $(\widetilde V, \widetilde U)$ to the problem \eqref{ci system}.

Following \cite{chiodaroli2015global} (and also \cite{de2009euler}), we will achieve \eqref{ci eqn} and \eqref{ci constraint =} by first considering the relaxed condition
\begin{equation}\label{ci constraint <}
{(\vo + \widetilde V) \otimes (\vo + \widetilde V) \over \ro} - (\uo + \widetilde U) < {|\vo|^2 \over n\ro} \id + \told. 
\end{equation}

Define the set 
\begin{equation}\label{defn X_0}
X_0 := \LCB (\widetilde V, \widetilde U) \in C^\infty_c(P; \R^n \times \sym_0): \ (\widetilde V, \widetilde U) \text{ solves } \eqref{ci eqn} \text{ and } \eqref{ci constraint <} \RCB.
\end{equation}
Obviously $X_0$ is nonempty since $0 \in X_0$ thanks to \eqref{def U_0} and $R_0 > 0$. Then we consider $X$ to be the closure of $X_0$ in the $L^\infty$ weak-$\ast$ topology. The metrizability of such a topology is ensured by the boundedness (in weak-$\ast$) of $X$ in $L^\infty$, and hence it generates a complete metric space $(X, d)$. Since elements of $X$ solve \eqref{ci eqn}, therefore the goal is to show that the saturation \eqref{ci constraint =} holds on a residual set so that a 
Bair\'e category argument applies. 

The main result of this section is the following.
\begin{lemma}\label{lem conv int}
Let $(\ro, V_0, R_0) \in C^0(\R^n\times \R; \R \times \R^n \times \sym)$ be given with $\ro $ satisfying \eqref{bounds rho} and $R_0$ being positive definite in some open set $P \subset \R^n\times \R$. Let $U_0$ be given as in \eqref{def U_0}. There exist infinitely many $(\tv, \tu) \in L^\infty(\R^n\times \R; \R^n \times \sym_0)$ which are compactly supported in $P$ and satisfy \eqref{ci eqn} and \eqref{ci constraint =}.
\end{lemma}
The proof of the above lemma relies on the following procedure which involves discretization of the problem and convex integration with a general non-diagonal Reynolds stress $R_0$ as performed in the previous section.  

Given $(\ro, V_0, R_0) \in C^0(\R^n\times \R; \R \times \R^n \times \sym)$ which satisfy the assumption of Lemma \ref{lem conv int}, and for a fixed $(\widetilde V, \widetilde U) \in X_0$ which is compactly supported in $P$, we denote
\begin{equation}\label{def M}
\begin{split}
M & := {|\vo|^2 \over n\ro} \id + \told - {(V_0 + \widetilde V) \otimes (V_0 + \widetilde V) \over \ro} + U_0 + \widetilde U, \\
\eigenmin(M) & := \text{ the smallest eigenvalue of } M. 
\end{split}
\end{equation}
From \eqref{ci constraint <} we see that $M > 0$ is positive definite on $P$, and hence $\eigenmin(M)> 0$ on $P$. Let us consider $\Omega_1\subset\Omega$ two compact subsets of $P$ such that  for $\delta_1>0$ small enough any cube of size $\delta_1$ centered in $\Omega_1$ is included in $\Omega$, and such that
$$
\int_{\Omega_1} \trace  M \,dx\,dt\geq \frac{1}{2} \int_P  \trace  M \,dx\,dt.
$$
Because $\Omega$ is compact and $M$ and $R_0$ are continuous on $P$, we have
\begin{equation}\label{min eigen}
\lambda_* := \min_{\overline\Omega} [\min (\eigenmin(M), \eigenmin(R_0))] > 0. 
\end{equation}
With this setup, we first prove the following lemma.
\begin{lemma}[$L^1$-coercivity]\label{lem beta ineq}
For any $(\widetilde V, \widetilde U) \in X_0$, there exists a sequence $\left\{(\tv_i, \tu_i)\right\} \subset X_0$ converging weak-$\ast$ to $(\tv, \tu)$ such that
\begin{equation}\label{total L1 coercive}
\LN \tv_i - \tv \RN_{L^1(P)} \ge {c_0 \over \Lambda} \int_P \textup\trace M \,dxdt
\end{equation}
for some geometric constant $c_0 >0$, where $\Lambda$ gives the bounds for $\ro$ as in \eqref{bounds rho}.
\end{lemma}
%
\begin{proof}
{\bf Step 1.} The idea is to perform a discretization. For that, let's first consider a localized problem. Take a fixed point $(x_0, t_0) \in \Omega_1$ and choose a sufficiently small open cube $Q$ centered at $(x_0, t_0)$ (especially of size smaller than $\delta_1$). Denote 
\begin{equation}\label{averages}
\begin{split}
\overline V := \fint_Q (V_0+\widetilde V)\,dxdt, \quad \overline U := \fint_Q (U_0+\widetilde U)\,dxdt, \quad \overline R_0 :=   \fint_Q R_0\,dxdt.
\end{split}
\end{equation}
Consider
\begin{equation}\label{def M_Q}
\overline M_Q := {C_Q \over n} \id +  \overline R_Q - {\overline V \otimes \overline V \over \underline\rho } + \overline U \in \sym, 
\end{equation}
where
\begin{equation}\label{def C_Q}
\underline\rho  :=\min_Q \rho_0, \quad C_Q := \min_{Q} \frac{ {|\vo|^2} }{\rho_0}\geq 0, \quad \overline R_Q :=  - {{\lambda_*} \over {16 n}}\id+\overline R_0.
\end{equation}
The uniform continuity of $\ro, \vo, \told$, $(\tv, \tu)$ in $\Omega$ implies that for any $\varepsilon > 0$ there exists some $\delta > 0$ independent of $Q$ such that whenever $|Q| < \delta$, the fluctuation of these quantities over $Q$ is smaller than $\varepsilon$. In particular choosing $\varepsilon$ small enough with respect to $\lambda_*$, we can ensure that for $\delta$ small enough,
\begin{equation*}
\overline R_Q > 0, \quad \sup_Q \| M - \overline M_Q \| < {\lambda_* \over {8n}}, \quad \sup_Q \|\overline R_0 - R_0\| < {\lambda_* \over 64n}, \quad \sup_Q \left|C_Q-|V_0|^2/\rho_0\right|\leq {\lambda_* \over 64n},
\end{equation*}
where $\|\cdot\|$ is the standard matrix norm.
Together with  \eqref{def C_Q} , we get
\begin{equation}\label{bound M_Q}
{C_Q \over n} \id + \overline R_Q < {|\vo|^2 \over n\ro}\id + R_0 - {\lambda_* \over 32n}\id,
\end{equation}
and
\begin{equation}\label{M bound}
\overline M_Q = M + (\overline M_Q - M) > M - {\lambda_* \over {8n} }\id > 0, \quad \text{ and hence } \quad  \trace \overline M_Q > {1\over 4}   \trace M \text{ on } Q.
\end{equation}
Consider the rescaled set $\mathring Q := \{ (x, t/\sqrt{\underline\rho }): \ (x, t) \in Q \}$. We now consider Proposition \ref{prop L^1 coercive} with the constants 
\[
(V_0, U_0, R_0, C_0) = \LC \frac{\overline V}{\sqrt{\underline\rho }}, \overline U, \overline R_Q, C_Q \RC.
\]
Denote $X_0^Q$ the set $X_0^c$ defined in (\ref{X_0 const}) on the set $\mathring Q$ (instead of $P$).  Thanks to \eqref{def M_Q} and \eqref{M bound}, $(0,0)\in X_0^Q$. 
Therefore, from Proposition  \ref{prop L^1 coercive}, there exists a sequence $(\mathring V_i,\mathring U_i)\in X^Q_0$ converging weakly to 0, and such that for every $i$:
\begin{eqnarray*}
&&\|\mathring V_i\|_{L^1(\mathring Q)}\geq c_1 (C_Q+\trace  \overline R_Q-|\overline V|^2)|\mathring Q|\geq c_1 (\trace  \overline M_Q)|\mathring Q|.
\end{eqnarray*}
Since $(\mathring V_i, \mathring U_i)\in X^Q_0$, it verifies   \eqref{ci eqn} and 
$$
\left(\frac{\overline V}{\sqrt{\underline\rho }}+\mathring V_i \right)\otimes \left(\frac{\overline V}{\sqrt{\underline\rho }}+\mathring V_i \right)-\overline  U+\mathring U_i< \frac{\overline  C_Q}{n}\id+  \overline R_Q.
$$
Consider the change of variable  $(\mathring V_i, \mathring U_i)(x,t) := \LC V_i/\sqrt{\underline\rho }, U_i \RC(x, t\sqrt{\underline\rho })$. The functions $(V_i, U_i)$ are now compactly supported in $Q$, and they still verify  \eqref{ci eqn} and converges weakly to 0. Moreover we have on $Q$ the following list of inequality, where we use the definition of $\underline\rho $ for the first inequality, and   Remark \ref{rem size} which ensures that the constant $C$ is independent of the sequence $V_i$ for the second one. The constant $C$ being fixed, we can get the third inequality by taking $\delta$ even smaller if needed. The last inequality follows from  \eqref{def C_Q}.
\begin{eqnarray*}
&& \frac{(V_0+\widetilde V+V_i)\otimes (V_0+\widetilde V+V_i)}{\rho_0}-(U_0+\widetilde{U}+U_i)\\
&&\quad \leq  \frac{(V_0+\widetilde V+V_i)\otimes (V_0+\widetilde V+V_i)}{\underline\rho }-(U_0+\widetilde{U}+U_i)\\
&&\quad \leq  \frac{(\overline V +V_i)\otimes (\overline V +V_i)}{\underline\rho }-(\overline U+U_i)+ C(|V_0+\widetilde V-\overline V_0|+|U_0+\widetilde U-\overline U_0|)\id\\
&&\quad \leq  \frac{\overline  C_Q}{n}\id+  \overline R_Q+ \frac{\lambda_*}{64n}\id\\
&&\quad < \frac{|V_0|^2}{n\rho_0}\id +R_0.
\end{eqnarray*}
Hence, $V_i+\widetilde V\in X_0$. And from the change of variables and \eqref{M bound}:
\begin{equation}\label{punch3}
\|V_i\|_{L^1(Q)}\geq  c_1 \underline\rho (\trace \overline M_Q)|Q|
\geq \frac{c_1}{4\Lambda}\int_Q \trace M \,dx\,dt.
\end{equation}

\medskip

{\bf Step 2.} Note again that the uniform continuity of $\ro, \vo, \told$, $(\tv, \tu)$ in $\Omega$ indicates that the size of $Q$ in {\bf Step 1} is independent of the choice of the point $(x_0, t_0)\in \Omega_1$. So we can repeat the argument of {\bf Step 1} at all points $(x, t) \in \Omega_1$. Taking $\delta<\delta_1$ small enough, we can consider the grid of points of $\R^n\times \R^+$: $(m_1 \delta, \cdot\cdot\cdot, m_n \delta, l\delta)$,  $l\in \N$, $(m_1,\cdot\cdot\cdot,m_n)\in \Z^n$. Consider a finite set of cubes of size $\delta$ with vertices on this grid covering  $\Omega_1$. Note that they have  non-overlapping interiors, and are all subsets of $\Omega$. Denote this list of cubes $\{Q_k: k=1,\cdot\cdot,N\}$. For each $k$, we denote $\{(V_i^k,U_i^k)\}_{k\in \N}$ the sequence of functions compactly supported in $Q_k$ defined in {\bf Step 1}, and define in $P$:
\[
\tv_i := \tv +\sum^N_{k=1} V_i^k, \qquad \tu_i := \tu +\sum^N_{k=1} U_i^k. 
\]
For $i$ fixed, all the  $(V_i^k, U_i^k)$ for $k=1\cdot\cdot\cdot N$ have disjoint supports. So from Step 1, $\tv_i\in X_0$. For all $k$ fixed, $V^k_i$ converges weakly to 0, so $\tv_i$ converges weakly to 0 as $i \to \infty$. 
Finally, from 
 \eqref{punch3}, we get
 $$
 \|\tv_i-\tv\|_{L^1(P)}\geq \sum_{k=1}^N \|V^k_i\|_{L^1(Q_k)}\geq \frac{c_1}{4\Lambda}\int_{Q_k} \trace  M\,dx\,dt\geq  \frac{c_1}{4\Lambda}\int_{\Omega_1} \trace  M\,dx\,dt\geq  \frac{c_1}{8\Lambda}\int_{P} \trace  M\,dx\,dt.
 $$
This completes the proof of the lemma.
\end{proof}

A direct consequence of the above lemma is the following.
\begin{proposition}[Points of continuity of the identity map]\label{prop continuity}
Let $(\widehat V, \widehat U) \in X$ be a point of continuity of the identity map $I$ from $(X,d)$ to $L^2(\R^n \times \R)$. Then $(\widehat V, \widehat U)$ satisfies \eqref{ci constraint =}.
\end{proposition}
%
\begin{proof}
By definition, there exists a sequence $(V_j, U_j) \in X_0$ converging weak-$\ast$ to $(\widehat V, \widehat U)$ with the property that $V_j \to \widehat V$ strongly in $L^2(P)$, and hence strongly in $L^1_{\text{loc}}(P)$. Lemma \ref{lem beta ineq} implies that for each $(V_j, U_j)$ one may find a sequence $\{(V_{j,i}, U_{j,i})\}$ converging weak-$\ast$ to $(V_j, U_j)$, satisfying \eqref{ci eqn}, and 
\begin{equation*}
\LN V_{j,i} - V_j \RN_{L^1(P)} \ge 
{c_1 \over 8 \Lambda} \int_{P} \trace M_j \,dxdt
\end{equation*}
where  $M_j$ are given in \eqref{def M} with $(\tv, \tu)$ being replaced by $(V_j, U_j)$. 
Applying a diagonal argument we obtain a subsequence $(V_{j, i(j)}, U_{j, i(j)})$ that converges weak-$\ast$ to $(\widehat V, \widehat U)$ and such that
\begin{equation*}
\lim_{j} \LN V_{j,i(j)} - \widehat V \RN_{L^1(P)} \ge {c_1 \over 8 \Lambda} \lim_j \int_{P} \trace M_j \,dxdt,
\end{equation*}
which implies that
\begin{equation}\label{limit M}
\lim_j \int_{P} \trace M_j \,dxdt = 0.
\end{equation}

Consider
\[
\widehat M := {|\vo|^2 \over n\ro} \id + \told - {(V_0 + \widehat V) \otimes (V_0 + \widehat V) \over \ro} + U_0 + \widehat U.
\]
We know that $M_j \rightharpoonup^* \widehat M$. Since $M_j > 0$, it follows that $\widehat M \ge 0$ a.e.. From \eqref{limit M} and the fact that  $\trace M_j \rightharpoonup^* \trace \widehat M$ we conclude that
\begin{equation*}
\lim_j \int_{P} \trace M_j \,dxdt = \int_{P} \trace \widehat M \,dxdt = 0.
\end{equation*}
Therefore $\trace \widehat M = 0$ a.e., and thus $\widehat M = 0$ a.e., which means $(\widehat V, \widehat U)$ satisfies \eqref{ci constraint =}.
\end{proof}

\begin{proof}[Proof of Lemma \ref{lem conv int}]
With the help of Proposition \ref{prop continuity}, Lemma \ref{lem conv int} follows from a Bair\'e category argument.
\end{proof}

From Lemma \ref{lem conv int} we immediately obtain the following proposition.
\begin{proposition}[Reduction to subsolutions]\label{prop reduction}
If \eqref{isentropic equivalent} has a solution $(\ro, \vo, \told) \in C^0(\R^n\times \R; \R \times \R^n \times \sym)$ with $\ro > 0$ and $\told > 0$ on a set $P$ of positive measure, then there are infinitely many bounded solutions $(\rho, v)$ to \eqref  {isentropic euler} with $\rho = \ro$.
\end{proposition}
\begin{proof}
From Lemma \ref{lem conv int} we know that under the assumption of Proposition \ref{prop reduction} one can find infinitely many bounded solutions $(\tv, \tu) \in C^\infty_c(P; \R^2 \times \sym_0)$ satisfying \eqref{ci eqn} and \eqref{ci constraint =}.

Now for any such $(\tv, \tu)$ define 
\[
\vn := \vo + \tv.
\]
Then 
\[
\partial_t \ro + \Div\ \vn = \partial_t \ro + \Div\ \vo + \Div\ \tv = 0.
\]
Applying \eqref{ci constraint =} and the definition \eqref{def U_0} that $\displaystyle U_0 := {\vo \otimes \vo \over \ro} - {|\vo|^2 \over n\ro} \id$ we have
\begin{align*}
& \partial_t \vn + \Div \LC {\vn \otimes \vn \over \rho_0} + p(\rho_0) \id \RC \\
= \ & \partial_t \vo + \partial_t \tv + \Div \LC \uo + \tu + {|\vo|^2 \over 2 \ro} \id + R_0 + p(\rho_0) \id \RC \\
= \ & \partial_t \vo + \Div \LC {\vo \otimes \vo \over \ro} + p(\rho_0) \id + R_0 \RC + \partial_t \tv + \Div \tu \\
= \ & 0.
\end{align*}
Taking $\displaystyle v := {\vn \over \ro}$ we see that $(\ro, v)$ solves \eqref{isentropic euler}.
\end{proof}
\begin{remark}\label{rk energy density}
At the energy level, taking trace in \eqref{ci constraint =} we see that after convex integration
\be\label{ci energy density}
\frac{\LV \vn \RV^2}{\ro} = \frac{\LV \vo \RV^2}{\ro} + \trace R_0.
\ee
This is to say, the `defect energy' of the subsolution due to the Reynolds tensor is injected into the weak Euler solutions through the convex integration.
\end{remark}
\begin{remark}\label{rk energy}
Notice that when $(\ro, \vo, R_0)$ is a piece-wise constant solution, like the ones constructed in \cite{chiodaroli2015global}, it follows from \eqref{ci energy density} that 
\begin{equation*}
\begin{split}
\partial_t E(\ro, \vn) & + \Div \LB \LC E(\ro, \vn) + p(\ro) \RC \vn \RB \\
= \ & \Div \LB \LC E(\ro, \vo) + {1\over 2} \trace R_0 + p(\ro) \RC \LC \vo + \tv \RC \RB = 0,
\end{split}
\end{equation*}
leading to a local energy balance.
\end{remark}

Proposition \ref{prop reduction} motivates us to define the following notion of subsolutions.
\begin{definition}[Subsolutions]\label{def subsoln}
A {\it subsolution} to the isentropic Euler system \eqref{isentropic euler} is a triple $(\ro, \vo, \told) \in C^0(\R^n\times \R; \R \times \R^n \times \sym)$ that solves \eqref{isentropic equivalent} with $\ro > 0$ and $\told > 0$ on a set $P$ of positive measure.
\end{definition}

\section{Application to incompressible flows}\label{sec incomp}

In this section we will apply the general scheme developed in Section \ref{sec ci} to treat the incompressible Euler equations. In this case, the system \eqref{eqn subsoln} for subsolutions changes to
\begin{equation}\label{eqn subsoln incomp}
\left\{
\begin{split}
\partial_t v + \Div (v \otimes v + p\id + R) & = 0, \\
\Div v & = 0,
\end{split}\right.
\end{equation}
where now density is take to be $\rho \equiv 1$ and the pressure $p$ becomes the Lagrange multiplier due to the incompressibility constraint. 

Our goal is to construct a large set of `wild' initial data of the incompressible Euler equations so that each such datum (1) generates infinitely many weak solutions and (2) those weak solutions satisfy the energy criterion. We will appeal to our convex integration scheme to handle the first part, provided that we are able to find a subsolution to \eqref{eqn subsoln incomp} with the Reynolds stress $R$ being positive definite. Regarding (2), we will need to ensure that the construction of the subsolutions is consistence with the energy law. 

\subsection{Energy compatible subsolutions}\label{subsec e-subsoln incomp}

In this subsection we precisely define the class of subsolutions we need for the convex integration and provide a way to construct them.
\begin{definition}[Energy compatible subsolutions]\label{defn e-subsoln incomp}
Let $\IE, T >0$, and $\mathcal{M}^+(\T^n; \sym)$ be the set of finite symmetric positive semidefinite matrix-valued (signed) Borel measures. We say that 
\[
(v, R) \in L^\infty(\R_+; L^2(\T^n)) \times L^\infty_{w^*}(\R_+; \mathcal{M}^+(\T^n; \sym))
\]
is an $(\IE, T)$-energy compatible subsolution of the incompressible Euler equations if the following conditions are satisfied
\begin{enumerate}[label=\rm(I\arabic*)]
\item \textup{(Existence of pressure)}\label{es pressure} 
   There exists some $p \in \mathcal{D}'(\R_+ \times \T^n)$ such that \eqref{eqn subsoln incomp} is satisfied in the sense of distribution on $\R_+ \times \T^n$.
\item \textup{(Short-time energy saturation)}\label{es saturate}
   For almost every $t \in [0, T]$ it holds that
   \begin{equation*}
   \frac12 \int_{\T^n} \LC {|v|^2} + \textup{\trace} R \RC \,dx = \IE.
   \end{equation*}
\item \textup{(Energy inequality)}\label{es ineq}
   For almost every $t \ge T$ it holds that
   \begin{equation*}
   \frac12 \int_{\T^n} \LC {|v|^2} + \textup{\trace} R \RC \,dx \le \IE.
   \end{equation*}  
\end{enumerate}
\end{definition}

One can show in the following proposition that for a smooth initial data, an energy compatible subsolution can be obtained through a classical vanishing viscosity limit. 
\begin{proposition}\label{prop e-subsoln incomp}
Let $v^0 \in C^1(\T^n)$ be divergence-free and for every $\nu > 0$, consider $v_\nu$ the (global) Leray solution to the Navier--Stokes equation with initial data $v^0$ which is divergence-free. Denote $\IE := \frac12 \int_{\T^n} |v^0|^2 \,dx$. Then there exist a $T>0$ and an $(\IE, T)$-energy compatible subsolution $(v, R)$ of the incompressible Euler equations such that up to a subsequence
\[
v_\nu \rightharpoonup v \ \ \text{in }\ \mathcal{D}' \quad \text{as }\ \nu \to 0.
\]
Moreover, $v$ is a Lipschitz solution to the Euler equation on $[0, T]$ with $v|_{t = 0} = v^0$.
\end{proposition}
\begin{proof}
From classical energy inequality for the Navier--Stokes equations we know that for any $\nu > 0$,
\be\label{NS ineq}
\frac12 \int_{\T^n} |v_\nu|^2 \,dx + \nu \int^t_0 \int_{\T^n} |\nabla v_{\nu}|^2 \,dxds \le \IE.
\ee
Hence there exists
\[
(v, R) \in L^\infty(\R_+; L^2(\T^n)) \times L^\infty_{w^*}(\R^n; \mathcal{M}^+(\T^n; \sym))
\]
such that up to a subsequence, as $\nu \to 0$,
\be\label{wk limit incomp}
v_\nu \rightharpoonup v, \quad v_\nu \otimes v_\nu \rightharpoonup v\otimes v + R \qquad \text{in }\ \mathcal{D}'.
\ee
Therefore passing to this limit in the Navier--Stokes equation gives \ref{es pressure}. It also implies that
\begin{equation*}
\frac12 \trace (v_\nu \otimes v_\nu) = \frac12 |v_\nu|^2 \rightharpoonup \frac12 \LC |v|^2 + \trace R \RC \quad \text{in }\ \mathcal{D}',
\end{equation*}
which, together with \eqref{NS ineq}, yields \ref{es ineq}.

Finally the local energy equality \ref{es saturate} follows from a classical weak-strong uniqueness argument. Recall that $v^0 \in C^1$, and hence there exists some $T>0$  and a Lipschitz solution $\underline{v}$ of the Euler equation on $[0, T]$. For any $\nu > 0$ and $t \le T$, following \cite{dafermos2005hyperbolic} we have
\begin{align*}
& \|v_\nu(t,\cdot) - \underline{u}(t, \cdot) \|^2_{L^2(\T^n)} + \nu \int^t_0 \int_{\T^n} |\nabla v_{\nu}(s,x)|^2 \,dxds \\
\le & \|\nabla \underline{v}\|_{L^\infty} \int^t_0 \int_{\T^n} |v_\nu - \underline{v}|^2 \,dxds + \nu \|\nabla \underline{v} \|_{L^2} \|\nabla v_\nu \|_{L^2} \\
\le & \|\nabla \underline{v}\|_{L^\infty} \int^t_0 \int_{\T^n} |v_\nu - \underline{v}|^2 \,dxds + \frac{\nu}{2} \int^t_0 \int_{\T^n} |\nabla v_{\nu}(s,x)|^2 \,dxds + 2\nu T \|\nabla \underline{v}\|_{L^\infty}.
\end{align*}
Thus using Gronwall and that $\underline{v}|_{t = 0} = u_\nu |_{t = 0}$ we have
\begin{equation*}
 \|v_\nu(t,\cdot) - \underline{u}(t, \cdot) \|^2_{L^2(\T^n)} \le 2\nu T \|\nabla \underline{v}\|_{L^\infty} e^{T \|\nabla \underline{v}\|_{L^\infty}} \to 0 \quad \text{as } \ \nu\to 0.
\end{equation*}
This together with \eqref{wk limit incomp} implies that 
\[
v = \underline{v} \quad \text{for }\ t \in [0, T]. 
\]
Therefore on $[0,T]$
\[
R \equiv 0 \quad \text{and} \quad \frac12 \int_{\T^n} |v|^2 \,dx = \IE,
\]
which gives \ref{es saturate}.
\end{proof}

\subsection{Convex integration}\label{subsec ci incomp}

In this subsection we explain how one applies the convex integration framework to produce from an energy compatible subsolution $(v, R)$ infinitely many weak Euler solutions emanating from an initial value which is, up to a defect energy, arbitrarily close to the initial value $v^0$ of the subsolution. 

Note that in order to apply our convex integration machinery, we need the Reynolds stress tensor $R$ to be positive definite. This would prohibit us from convex integrating the subsolution from the initial time directly since there is an energy `bump-up' at the initial time coming from $R$, as can be seen in \eqref{ci energy density}. Therefore the idea is to first convex integrate $(v, R)$ from its initial value for a short time period $[0, t_0]$, which generates a new `initial value' at some $\widetilde t \in (0, t_0)$ carrying the full energy. Then convex integrate on $[t_0, +\infty)$, and finally stick together the two pieces. 
\begin{theorem}\label{thm ci incomp}
Assume that $(v, R)$ is a smooth $(\IE, T)$-energy compatible subsolution of the incompressible Euler equations with initial value $(v^0, R^0)$ such that $v^0$ is divergence-free and $R(x, t) > 0$ is positive definite for every $(x,t) \in \T^n \times \R_+$. Then for any $\varepsilon > 0$, there exist infinitely many divergence-free initial values $\widetilde v^0_\varepsilon \in L^2(\T^n)$ such that 
\be\label{ci data cond1}
\frac12 \int_{\T^n} |\widetilde v^0_\varepsilon(x)|^2 \,dx = \IE \quad \text{and} \quad \int_{\T^n} |\widetilde v^0_\varepsilon - v^0|^2 \,dx < \varepsilon + \int_{\T^n} \textup{tr} R^0 \,dx,
\ee
and for each of such initial values there exist infinitely many $\widetilde v_\varepsilon \in L^\infty(\R_+; L^2(\T^n))$ which are global weak solutions to the incompressible Euler equations with $\widetilde v_\varepsilon |_{t = 0} = \widetilde v^0_\varepsilon$ and 
\begin{equation}\label{ci data cond2}
\frac12 \int_{\T^n} |\widetilde v_\varepsilon(t,x)|^2 \,dx \le \frac12 \int_{\T^n} |\widetilde v^0_\varepsilon(x)|^2 \,dx = \IE \quad \text{a.e. }\ t > 0.
\end{equation}
\end{theorem}
\begin{proof}
Since $(v, R)$ is smooth, so for any $\varepsilon > 0$ there exists some $t_0 < \frac{T}{2}$ such that $\forall \ t < t_0$,
\begin{align}
t_0 \|\nabla v^0\|_{L^\infty} \IE & < \frac{\varepsilon}{4}. \label{cond ci 1} 
\end{align}
The positivity of $R$ ensures that we can apply our convex integration result Proposition \ref{prop reduction} (with density being constant) on $[0, t_0]$. This provides infinitely many weak Euler solutions $\hat v$ on $[0, t_0]$ with
\[
\hat v |_{t = 0} = v^0, \quad \hat v(t_0, \cdot) = v(t_0, \cdot),
\]
and by \eqref{ci energy density},
\be\label{e-eq}
\frac12 \int_{\T^n} |\hat v(t,x)|^2 \,dx = \IE, \quad \text{a.e. }\ t \in [0, t_0].
\ee

Denote the set $\mathcal T \subset [0, t_0]$ to be the set of times such that the above equality holds. Then $\mathcal T$ depends on the solution $\hat v$, $0, t_0 \not\in \mathcal T$, and the measure $\mathcal{L}([0, t_0] \backslash \mathcal {T}) = 0$.

If there are only a finite number of $L^2$ functions $\{ v_1, \ldots, v_N\}$ such that for all the weak solutions $\hat v$ constructed above with the associated $\mathcal T$, $\hat v(t) \in \{ v_1(t), \ldots, v_N(t)\}$. 
Then by the weak continuity in time, we indeed have that 
\[
\hat v \equiv v_j \quad \text{on }\ [0, t_0/2] \ \text{for some } \ j \in \{1, \ldots, N \}.
\]
But this would in turn imply that we can only construct $N$ weak Euler solutions, which is a contradiction with the fact that our convex integration scheme can produce infinitely many weak solutions.

The above discussion allows us to choose infinitely many $\widetilde t \in \mathcal T$ to define infinitely many initial data 
\[
\widetilde v^0_\varepsilon := \hat v(\widetilde t, \cdot).
\]
Hence from the definition of $\mathcal T$ we see that $\frac12 \int_{\T^n} |\widetilde v^0_\varepsilon(x)|^2 \,dx = \IE$. Moreover from \ref{es saturate},
\begin{equation*}
\begin{split}
\frac12 \|\widetilde v^0_\varepsilon - v^0\|^2_{L^2} & = \frac12 \|\widetilde v^0_\varepsilon\|^2_{L^2} + \frac12\|v^0\|^2_{L^2} - \int_{\T^n} \widetilde v^0_\varepsilon \cdot v^0 \,dx \\
& = \IE + \LC \IE - \frac12\int_{\T^n} \trace R^0 \,dx \RC - \int_{\T^n} \widetilde v^0_\varepsilon \cdot v^0 \,dx
\end{split}
\end{equation*}
Since $v^0$ is divergence-free, by \eqref{cond ci 1} and \eqref{e-eq},
\begin{equation*}
\begin{split}
\LV \int_{\T^n} \hat v(t,x) \cdot v^0(x) \,dx - \|v^0\|^2_{L^2} \RV 
& \le \LV \int^t_0 \int_{\T^n} \partial_t \hat v \cdot v^0 \,dxds \RV \\
& \le \LV \int^t_0 \int_{\T^n} \nabla v^0 : (\hat v \otimes \hat v) \,dxds \RV \\
& \le 2t \|\nabla v^0\|_{L^\infty} \IE < \frac{\varepsilon}{2}.
\end{split}
\end{equation*}
Therefore from \ref{es saturate} it follows that
\begin{equation*}
\begin{split}
\frac12 \|\widetilde v^0_\varepsilon - v^0\|^2_{L^2} & = \frac12 \|\widetilde v^0_\varepsilon\|^2_{L^2} + \frac12\|v^0\|^2_{L^2} - \int_{\T^n} \widetilde v^0_\varepsilon \cdot v^0 \,dx \\
& = \frac12 \|\widetilde v^0_\varepsilon\|^2_{L^2} - \frac12\|v^0\|^2_{L^2} + \|v^0\|^2_{L^2} - \int_{\T^n} \widetilde v^0_\varepsilon \cdot v^0 \,dx \\
& < \IE - \LC \IE - \frac12\int_{\T^n} \trace R^0 \,dx \RC + \frac{\varepsilon}{2} \\
& = \frac12\int_{\T^n} \trace R^0 \,dx + \frac{\varepsilon}{2}.
\end{split}
\end{equation*}
This proves \eqref{ci data cond1}.

Now for each of such (infinitely many) initial values $\widetilde v^0_\varepsilon$ we define $\widetilde v_\epsilon$ on $\R_+ \times \T^n$ as follows:
\begin{equation*}
\widetilde v_\varepsilon(t, \cdot) = \left\{\begin{array}{ll}
\hat v(t + \widetilde t, \cdot), \quad & \ \text{ for } \ t \le t_0 - \widetilde t, \\
\breve v(t + \widetilde t, \cdot),  & \ \text{ for }\ t \ge t_0 - \widetilde t,
\end{array}\right.
\end{equation*}
where $\breve v$ is any weak Euler solution on $[t_0, +\infty)$ constructed by convex integrating the original energy compatible subsolution $(v, R)$ on $[t_0, +\infty)$. Thus we know that
\[
\hat v(t_0) = v(t_0) = \breve{v}(t_0). 
\]
This way we know that $\widetilde v_\varepsilon$ is a weak solution to the incompressible Euler equations on $\R_+ \times \T^n$. By construction we have $\widetilde v_\epsilon|_{t = 0} = \widetilde v^0_\varepsilon$, and  $\widetilde v_\epsilon$ satisfies \eqref{ci data cond2}.
\end{proof}

The above construction can be illustrated in the following diagram.
\begin{figure}[h]
  \includegraphics[page=1,scale=0.85]{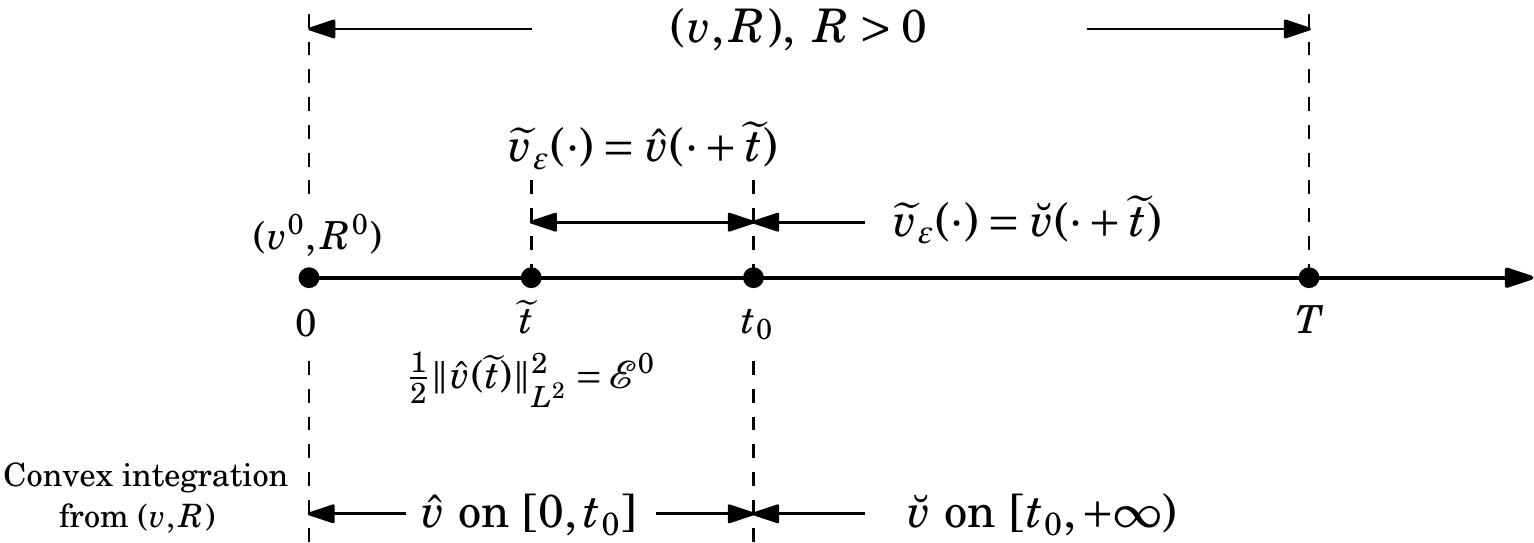}
  \caption{Double convex integration of an $(\IE, T)$-energy compatible subsolution with $R > 0$.}
  \label{fig convint}
\end{figure}

\subsection{Construction of smooth energy compatible strict subsolutions}\label{subsec strict subsoln}

From Theorem \ref{thm ci incomp} we see that in order to construct infinitely many weak solutions to the incompressible Euler equation \eqref{incomp euler system} with initial data being a small perturbation of a given $L^2$ function, it suffices to find an energy compatible subsolution $(v, R)$ in the sense of Definition \ref{defn e-subsoln incomp} which further satisfies that $R > 0$ is positive definite. We call such a $(v, R)$ an energy compatible {\it strict} subsolution.

In this subsection we introduce a convex combination formalism to produce such energy compatible strict subsolutions.

Let $(\Omega, \mu)$ be a probability space, that is, $\mu$ is a nonnegative measure on $\Omega$ such that $\mu(\Omega) = 1$.
\begin{lemma}\label{lem conv combo}
Fix $\IE, T > 0$. Let $(v, R)$ be a measurable function from $\R_+ \times \T^n \times \Omega$ to $\R^n \times \R^{n \times n}$ such that for a.e. $\omega \in \Omega$, $(v_\omega, R_\omega) := \LC v(\cdot, \omega), R(\cdot, \omega) \RC$ is an $(\IE, T)$-energy compatible subsolution to the incompressible Euler equations. Denote
\be\label{conv combo}
\overline v := \mathbb{E} (v_\omega), \qquad \overline R := \mathbb E(R_\omega) + \mathbb E(v_\omega \otimes v_\omega) - \overline v \otimes \overline v,
\ee
where 
\[
\mathbb E(f_\omega) := \int_\Omega f_\omega d\mu(\omega).
\]
Then $(\overline v, \overline R)$ is also an $(\IE, T)$-energy compatible subsolution.
\end{lemma}
\begin{remark}
This lemma says that the set of energy compatible subsolutions is closed under convex combination (discrete or continuous), and \eqref{conv combo} gives the explicit formula for the new `Reynolds tensor'.
\end{remark}
\begin{proof}[Proof of Lemma \ref{lem conv combo}]
Taking the expectation of \eqref{eqn subsoln incomp} and using the definition $\eqref{conv combo}$ it follows that $(\overline v, \overline R)$ and $\overline p := \mathbb E(p_\omega)$ satisfy \ref{es pressure}.

Note that for a.e. $t \in [0, T]$, by \eqref{conv combo}
\begin{equation*}
\begin{split}
\mathbb E\LC \frac12(|v_\omega|^2 + \trace R_\omega) \RC & = \frac12 \trace \mathbb E\LC v_\omega \otimes v_\omega + R_\omega \RC \\
& = \frac12 \trace \LC \overline v \otimes \overline v + \overline R \RC = \frac12(|\overline v|^2 + \trace {\overline R}).
\end{split}
\end{equation*}
Taking the expectation of \ref{es saturate} and \ref{es ineq} for $(v_\omega, R_\omega)$ proves \ref{es saturate} and \ref{es ineq} for $(\overline v, \overline R)$.
\end{proof}

Now we are ready to prove the main result of this subsection.
\begin{theorem}\label{thm smoothing and conv combo}
Given $\IE, T > 0$, let $(v, R)$ be an $(\IE, T)$-energy compatible subsolution to the incompressible Euler equations with initial data $v^0$ satisfying 
\[
\frac12 \int_{\T^n} |v^0|^2 \,dx = \IE. 
\]
Then for any $\varepsilon > 0$ there exists a smooth $(\IE, \frac{T}{2})$-energy compatible subsolution $(\widetilde v, \widetilde R)$ with 
\be\label{strict subsoln}
\widetilde R > 0, \qquad \frac12 \int_{\T^n} \LC |\widetilde v^0 - v^0|^2 + \textup{tr} \widetilde R^0 \RC \,dx < \varepsilon.
\ee 
\end{theorem}
\begin{remark}
Note that from the hypothesis we know that $R^0 \equiv 0$. This is verified by the subsolutions constructed from the vanishing viscosity limit, cf. Proposition \ref{prop e-subsoln incomp}. 
\end{remark}
\begin{proof}
First we know that $(\IE, T)$-energy compatible subsolutions $(v, R)$ with $\frac12 \int_{\T^n} |v^0|^2 \,dx = \IE$ are strongly continuous in $L^2$ at $t = 0$, that is, $\displaystyle \lim_{t\to 0} \|v(t) - v^0\|_{L^2} = 0$. 

Indeed, for any $\varepsilon > 0$, there is a smooth divergence-free function $v^0_\varepsilon \in C^\infty(\T^n)$ such that
\[
\|v^0 - v^0_\varepsilon\|^2_{L^2} < \frac{\varepsilon}{3}\IE,
\]
together with some constant $C_\varepsilon > 0$ such that
\[
\LV \frac{d}{dt} \int_{\T^n} v^0_\varepsilon(x) \cdot v(t,x) \,dx \RV \le \LV \int_{\T^n} \nabla v^0_\varepsilon : (v \otimes v + R) \,dx \RV \le C_\varepsilon.
\]
From this it follows that there exists some $t_0 > 0$ such that $C_\varepsilon t_0 < \varepsilon/3$. Hence for a.e. $t \in [0, t_0]$,
\begin{equation*}
\begin{split}
\frac12\|v(t, x) - v^0(x)\|^2_{L^2} & = \frac12\|v\|^2_{L^2} + \frac12\|v^0\|^2_{L^2} - \int_{\T^n} (v^0 - v^0_\varepsilon) \cdot v \,dx - \int_{\T^n} v^0_\varepsilon \cdot v \,dx \\
& < 2 \IE + \frac{\varepsilon}{3} - \int_{\T^n} v^0_\varepsilon \cdot v^0\,dx + \int^{t_0}_0 \frac{d}{ds}\int_{\T^n} v^0_\varepsilon \cdot v \,dx ds \\
& < \frac{\varepsilon}{3} + \int_{\T^n} (v^0_\varepsilon - v^0)\cdot v^0 \,dx + C_\varepsilon t_0  < \varepsilon.
\end{split}
\end{equation*}

Moreover for a.e. $t \ge 0$, recall from the definition of $(\IE, T)$-energy compatible subsolution that
\[
\frac12 \int_{\T^n} (|v|^2 + \trace R) \,dx \le \IE,
\]
from which we have
\[
\frac12 \int_{\T^n} \trace R\,dx \le \IE - \frac12\int_{\T^n} |v|^2 \,dx \to 0 \quad \text{as }\ t \to 0.
\]
Thus $\int_{\T^n} \trace R\,dx$ is continuous in time at $t = 0$ and 
\[
\int_{\T^n} \trace R^0 \,dx = 0.
\]

Introducing a standard mollifier $\varphi \in C^\infty(\R \times \T^n)$ with support in $[0, 1] \times \T^n$, $\varphi \ge 0$ and $\int \varphi = 1$. For $\alpha > 0$ we define 
\[
\varphi_\alpha(t, x) := \frac{1}{\alpha^{n+1}} \varphi\LC \frac{t}{\alpha}, \frac{x}{\alpha} \RC.
\]
Therefore we can find some $0 < \alpha < T/2$ small enough such that $\varphi_\alpha \ast (v^0, R^0)$
satisfy
\[
\|\varphi_\alpha \ast v^0 - v^0\|^2_{L^2} < \frac{\varepsilon}{2}, \qquad \int_{\T^n} \trace (\varphi_\alpha \ast R^0)\,dx < \frac{\varepsilon}{2}.
\]
Then consider $\Omega := [0, T/2] \times \T^n$, $d\mu(s, y) := \varphi_\alpha(s, y) ds dy$, and for any $\omega := (s, y) \in \Omega$ with $s < T/2$,
\[
(w, \mathcal R)(\cdot, \cdot, s, y) := \LC v, R \RC (\cdot + s, \cdot + y)
\]
on $[0, T/2] \times \T^n$. Therefore $(w, \mathcal R)(\cdot, \cdot, s, y)$ is an $(\IE, \frac{T}{2})$-energy compatible subsolution.

Note that
\[
\mathbb E(w) = \varphi_\alpha \ast w =: v_\alpha, \qquad  \mathbb E(\mathcal R) = \varphi_\alpha \ast R.
\]
So, thanks to Lemma \ref{lem conv combo}, $(v_\alpha, R_\alpha)$ is still an $(\IE, \frac{T}{2})$-energy compatible subsolution with
\[
R_\alpha := \varphi_\alpha \ast R + \varphi_\alpha \ast (v  \otimes v) - v_\alpha \otimes v_\alpha.
\]
It is obvious that $(v_\alpha, R_\alpha)$ is smooth and by shrinking $\alpha$ if necessary we have
\be\label{ic close}
\|v^0_\alpha - v^0\|^2_{L^2} + \int_{\T^n} \trace R^0_\alpha \,dx < \varepsilon.
\ee

Now define
\[
\lambda := \min\LC \frac{\varepsilon}{6\IE}, \frac12 \RC \in (0, 1).
\]
Consider $\Omega := \{1, 2\}$, $\mu$ being atomic as $\mu := \lambda \delta_{\omega = 1} + (1-\lambda) \delta_{\omega=2}$, and 
\begin{align*}
(w, \mathcal R)(t, x, 1) := \LC 0, \frac{2\IE}{n |\T^n|} \id \RC, \qquad  (w, \mathcal R)(t, x, 2) := \LC v_\alpha, R_\alpha \RC.
\end{align*}

Note that $\LC 0, \frac{2\IE}{n |\T^n|} \id \RC$ is an $(\IE, T)$-energy compatible subsolution. So for any $\omega \in \Omega$, $(w, \mathcal R)(\cdot, \cdot, \omega)$ is an $(\IE, \frac{T}{2})$-energy compatible subsolution and 
\[
\mathbb E(w) = (1- \lambda) v_\alpha, \qquad \mathbb E(\mathcal R) = \frac{2\lambda \IE}{n |\T^n|} \id + (1 - \lambda) R_\alpha.
\]
Therefore from Lemma \ref{lem conv combo},
\[
(\widetilde v, \widetilde R) := \big( (1- \lambda) v_\alpha, \lambda(1 - \lambda) v_\alpha \otimes v_\alpha + \mathbb E(\mathcal R) \big)
\]
is an $(\IE, \frac{T}{2})$-energy compatible subsolution. Moreover $(\widetilde v, \widetilde R)$ is smooth and
\[
\widetilde R \ge \frac{2\lambda \IE}{n |\T^n|} \id > 0.
\]

Finally we check the initial value by recalling \eqref{ic close}  to obtain
\begin{align*}
\|\widetilde v^0 - v^0\|^2_{L^2} & = \|v^0_\alpha - v^0 - \lambda v^0_\alpha\|^2_{L^2} < \varepsilon, \\
\int_{\T^n} \trace \widetilde R^0 \,dx & = 2 \lambda \IE + (1 - \lambda) \int_{\T^n} \trace R^0_\alpha \,dx + \lambda (1 - \lambda) \IE \\
& \le 3 \lambda \IE + \frac{1 - \lambda}{2} \varepsilon < \varepsilon.
\end{align*}
Putting together we obtain \eqref{strict subsoln}.
\end{proof}

\subsection{Density of wild data for incompressible Euler equations}\label{subsec density incomp}

With all of the above preparation, we are now in a position to prove  Theorem \ref{thm main incomp}.
\begin{proof}
We will first mollify $u^0$ to a smooth $u^0_\varepsilon \in C^\infty(\T^n)$ such that 
\begin{equation}\label{incomp ic 1}
\|u^0_\varepsilon - u^0\|_{L^2(\T^n)}^2 < \frac{\varepsilon}{9}.
\end{equation}
Denote $\IE_\varepsilon := \frac12 \int_{\T^n} |u^0_\varepsilon|^2 \,dx$. Applying Proposition \ref{prop e-subsoln incomp} yields the existence of some $T > 0$ and an $(\IE_\varepsilon, T)$-energy compatible subsolution $(u_\varepsilon, R_\varepsilon)$ with $u_\varepsilon|_{t = 0} = u^0_\varepsilon$, and $R_\varepsilon \equiv 0$ for $t \in [0, T]$.

Next we can apply Theorem \ref{thm smoothing and conv combo} so that $(u_\varepsilon, R_\epsilon)$ is upgraded to a smooth $(\IE_\varepsilon, \frac{T}{2})$-energy compatible subsolution $(\widetilde u_\varepsilon, \widetilde R_\varepsilon)$ with initial data $(\widetilde u^0_\varepsilon, \widetilde R^0_\varepsilon)$ and satisfying the properties
\begin{equation}\label{incomp ic 2}
\widetilde R_\varepsilon > 0, \quad \text{and} \quad \|\widetilde u^0_\varepsilon - u^0_\varepsilon\|_{L^2(\T^n)}^2 + \int_{\T^n} \trace \widetilde R^0_\varepsilon \,dx < \frac{\varepsilon}{9}.
\end{equation}

Finally the positivity of $\widetilde R^0_\varepsilon$ allows us to use Theorem \ref{thm ci incomp}. This produces infinitely many $v^0 \in L^2(\T^n)$ each of which induces infinitely many weak solutions in the sense of Definition \ref{def wk soln incomp}. Moreover from \eqref{ci data cond1} we see that
\begin{equation}\label{incomp ic 3}
\|v^0 - \widetilde u^0_\varepsilon\|_{L^2(\T^n)}^2 < \frac{\varepsilon}{9} + \int_{\T^n} \trace \widetilde R^0_\varepsilon \,dx.
\end{equation}
Putting together \eqref{incomp ic 1}--\eqref{incomp ic 3} it follows that
\begin{equation*}
\begin{split}
\|v^0 - u^0\|_{L^2(\T^n)}^2 & \le 3 \LC \|v^0 - \widetilde u^0_\varepsilon\|_{L^2(\T^n)}^2 + \|\widetilde u^0_\varepsilon - u^0_\varepsilon\|_{L^2(\T^n)}^2 + \|u^0_\varepsilon - u^0\|_{L^2(\T^n)}^2 \RC \\
& < 3 \LC \frac{\varepsilon}{9} + \int_{\T^n} \trace \widetilde R^0_\varepsilon \,dx + \|\widetilde u^0_\varepsilon - u^0_\varepsilon\|_{L^2(\T^n)}^2 + \frac{\varepsilon}{9} \RC = \varepsilon, 
\end{split}
\end{equation*}
leading to \eqref{incomp ic}, and therefore the proof is completed. 
\end{proof}

\section{Application to compressible flows}\label{sec comp}

As a second application of our convex integration scheme, we consider the problem of constructing infinitely many admissible weak solutions of the compressible Euler equations \eqref{isen euler system} in the sense of Definition \ref{def wk soln comp}. The basic strategy is similar to the incompressible case. For a given data in the energy space, we first smooth it out, and build up an energy compatible subsolution via vanishing viscosity. Before applying our convex integration scheme, we need further regularize the energy compatible subsolution, and moreover to enhance the defect $R$ to be positive definite. 

Compared with the incompressible case, a notable difference for the compressible system is the additional contribution to the defect from the density variable. Because of this, we will modify our definition of the energy compatible subsolutions as follows. 
\begin{definition}[Energy compatible subsolutions]\label{defn e-subsoln comp}
Let $\gamma > 1$, $\IE, T >0$, and $\mathcal{M}^+$ be the set of finite nonnegative (signed) Borel measures on $\T^n$. We say that 
\[
(\rho, V, \mathcal R, r) \in L^\infty(\R_+; L^\gamma) \times L^\infty(\R_+; L^{\frac{2\gamma}{\gamma + 1}}) \times L^\infty_{w^*}(\R_+; \mathcal{M}^+(\T^n; \sym)) \times L^\infty_{w^*}(\R_+; \mathcal{M}^+)
\]
is an $(\IE, T)$-energy compatible subsolution of the compressible Euler equations if the following conditions are satisfied.
\begin{enumerate}[label=\rm(C\arabic*)]
\item \textup{(Weak subsolution)}\label{es subsoln comp} 
$\rho \ge 0$, $V = 0$ whenever $\rho = 0$, and the following system 
\begin{equation}\label{definition of sub-solution}
\left\{
\begin{split}&\rho_t+\Div V=0,
\\& V_t+\Div \left(\frac{V \otimes V}{\rho}+ \mathcal R+r \id+p(\rho) \id\right)=0;
\end{split}\right.
\end{equation}
holds in the sense of distribution on $\R_+ \times \T^n$.
\item \textup{(Short-time energy saturation)}\label{es saturate comp}
   For almost every $t \in [0, T]$ it holds that
   \begin{equation*}
   \int_{\T^n} \LC E(\rho, V) + \frac12 \textup{\trace} \mathcal R + \frac{r}{\gamma - 1} \RC \,dx = \IE,
   \end{equation*}
   where
   \begin{equation}\label{entropy}
   E(\rho, V) := \frac{|V|^2}{2\rho} + \frac{p(\rho)}{\gamma - 1}
   \end{equation}
   is the associated entropy.
\item \textup{(Energy inequality)}\label{es ineq comp}
   For almost every $t \ge T$ it holds that
   \begin{equation*}
   \int_{\T^n} \LC E(\rho, V) + \frac12 \textup{\trace} \mathcal R + \frac{r}{\gamma - 1} \RC \,dx \le \IE.
   \end{equation*}  
\end{enumerate}
\end{definition}

Analogous to Proposition \ref{prop e-subsoln incomp}, we have the following result ensuring the existence of the compressible energy compatible solutions. 
\begin{proposition}\label{prop e-subsoln comp}
Let $\rho^0, v^0 \in C^1(\T^n)$, $\rho^0 > 0$. For any $\gamma > 1$ and $\nu > 0$ consider $(\rho_\nu, V_\nu)$ the (global) weak solution to the following compressible Navier--Stokes equation constructed in \cite{VY2016,bresch2019global} with initial data $(\rho^0, V^0 := \rho^0 v^0)$ 
\begin{equation}
\label{Com NS}
 \left\{
\begin{split}&
\partial_t\rho_{\nu}+\Div V_{\nu}=0,
\\&\partial_t V_{\nu}+\Div \left(\frac{V_{\nu} \otimes V_{\nu}}{\rho_{\nu}}+ p(\rho) I_n\right)= \Div \LC \sqrt{\nu \rho_\nu} \mathbb{S}_\nu \RC, 
\end{split}\right.
\end{equation}
where
\[
\sqrt{\nu \rho_\nu} \mathbb{S}_\nu := \nu\rho_{\nu}\mathbb{D} v_{\nu} \quad \text{with} \quad \mathbb{D} v_{\nu} := \LC \frac{\nabla v_\nu + \nabla^T v_\nu}{2} \RC \quad \text{and} \quad V_\nu = \rho_\nu v_\nu.
\]
Set 
\[
\IE := \int_{\T^n} E(\rho^0, V^0) \,dx = \int_{\T^n} \LC \frac{|V^0|^2}{2\rho^0} + \frac{p(\rho^0)}{\gamma - 1} \RC \,dx. 
\]
Then there exist a $T>0$ and an $(\IE, T)$-energy compatible subsolution $(\rho, V, \mathcal R, r)$ of the compressible Euler equations such that up to a subsequence
\[
(\rho_\nu, V_\nu) \rightharpoonup (\rho, V) \ \ \text{in }\ \mathcal{D}' \quad \text{as }\ \nu \to 0.
\]
Moreover, $(\rho, V)$ is a Lipschitz solution to the compressible Euler equation on $[0, T]$ with $(\rho, V)|_{t = 0} = (\rho^0, V^0)$.
\end{proposition}
\begin{remark}
Here we follow the notation of \cite{LV2018} to use $\mathbb S_\nu$ in the dissipation term since the a priori estimates do not seem to be sufficient to define $\nabla v_\nu$.
\end{remark}
\begin{proof}
Recall from \cite{VY2016,bresch2019global} that system \eqref{Com NS} admits a global weak solution $(\rho_\nu, V_\nu)$ with $\rho_\nu \ge 0$ and $(\rho_\nu, V_\nu)|_{t = 0} = (\rho^0, V^0)$ and for a.e. $t \ge 0$, 
\begin{equation}\label{comp energy ineq}
\int_{\T^n} E(\rho_\nu, V_\nu) \,dx + \int^t_0\int_{\T^n} |\mathbb{S}_{\nu}|^2 \,dxds  \le \int_{\T^n} E(\rho^0, V^0) \,dx.
\end{equation}
where the dissipation term $\int^t_0\int_{\T^n} |\mathbb{S}_{\nu}|^2 \,dxds$ is formally $\nu \int^t_0\int_{\T^n} \rho_\nu |\mathbb{D} v_{\nu}|^2 \,dxds$. \footnote{Although \eqref{comp energy ineq} is not explicitly given in \cite{VY2016}, one may easily obtain it from replacing the term $\int^t_0\int_{\Omega} \rho_\nu |\mathbb{D} v_{\nu}|^2 \,dxds$ in \cite[(1.7)]{VY2016} by $\int^t_0\int_{\Omega} |\mathbb{S}_{\nu}|^2 \,dxds$, taking the limit as $\kappa \to 0$, then $r_0, r_1 \to 0$, and using the weak lower semicontinuity of $\int^t_0\int_{\Omega} |\mathbb{S}_{\nu}|^2 \,dxds$.}
This clearly yields that as $\nu \to 0$, up to a subsequence, 
\begin{equation*}
(\rho_\nu, V_\nu) \rightharpoonup (\rho, V) \ \ \text{weakly in }\ L^{\infty}(\R_+;L^{\gamma}(\mathbb T^n)) \times L^{\infty}(\R_+; L^{\frac{2\gamma}{\gamma+1}}( \mathbb T^n)),
\end{equation*}
which defines
\begin{equation}\label{def Rr}
\mathcal R := \lim_{\nu\to0}\frac{V_{\nu}\otimes V_{\nu}}{\rho_{\nu}}-\frac{V \otimes V}{\rho}, \qquad r := \lim_{\nu\to 0}p(\rho_{\nu})- p(\rho) \quad \text{in }\ \mathcal D'.
\end{equation}
Thanks to convexity we know that $(\mathcal R, r) \in L^\infty_{w^*}(\R^n; \mathcal{M}^+(\T^n; \sym)) \times L^\infty_{w^*}(\R^n; \mathcal{M}^+)$. Therefore \ref{es subsoln comp} follows by sending $\nu \to 0$ in \eqref{Com NS}.

From \eqref{def Rr} we see that
\[
\frac{|V_\nu|^2}{\rho_\nu} \rightharpoonup \frac{|V|^2}{\rho} + \trace \mathcal R, 
\]
which, together with \eqref{comp energy ineq}, implies \ref{es ineq comp}.

Similar as in Proposition \ref{prop e-subsoln incomp}, \ref{es saturate comp} follows from a weak-strong uniqueness argument. For the sake of completeness we will briefly sketch the idea. Since $(\rho^0, v^0) \in C^1(\T^n)$, we know that there exists a unique classical solution $(\rho_E, v_E)$ to the Euler equations on some time interval $[0, T]$ with $\rho_E > 0$.

Denote
\[
U_\nu := (\rho_\nu, V_\nu), \qquad U_E := (\rho_E, V_E := \rho_E v_E)
\]
The entropy for the Euler system is given by $E(U)$ defined in \eqref{entropy}, which is regular and strictly convex for $U \in \mathcal V := \R_+ \times \R^n$. Recall the definition of the relative entropy $E(\cdot | \cdot): \mathcal V \times \mathcal V \to \R$
\[
E(U_1 | U_2) = E(U_1) - E(U_2) - E'(U_2) \cdot (U_1 - U_2).
\]
The convexity of $E$ ensures that the relative entropy $E(U_1 | U_2)$ defines a pseudo-distance on $\mathcal V$, and hence $E(U_1 | U_2) = 0$ if and only if $U_1 = U_2$.

Since $U^0_\nu = U^0_E$, direct computation yields that for $t \in [0, T]$,
\begin{align*}
\int_{\T^n} E(U_\nu | U_E) \,dx &\le - \int^t_0 \int_{\T^n} \nabla_x E'(U_E) : \LC 0, \frac{(V_\nu - V_E) \otimes (V_\nu - V_E)}{\rho_\nu} + \frac{p(\rho_\nu | \rho_E)}{\gamma - 1} \id \RC \,dx ds \\
& \quad - \int^t_0\int_{\T^n} |\mathbb{S}_\nu|^2 \,dx ds + \int^t_0 \int_{\T^n} \sqrt{\nu \rho_\nu} \nabla v_E : \mathbb{S}_\nu \,dx ds \\
& \le C_\nu \int^t_0 \int_{\T^n} E(U_\nu | U_E) \,dx ds + 2 \nu \int^t_0 \int_{\T^n} \rho_\nu |\nabla v_E|^2 \,dx ds \\
& \le C_\nu \LC \int^t_0 \int_{\T^n} E(U_\nu | U_E) \,dx ds + \nu \RC.
\end{align*}
Sending $\nu \to 0$ and applying Gronwall it follows that
\[
(\rho, V) = (\rho_E, V_E) \quad \text{for} \quad t \in [0, T].
\]
This further implies that $\mathcal R \equiv 0$ and $r \equiv 0$ for a.e. $t \in [0, T]$.
\end{proof}

\subsection{Convex integration}\label{subsec ci comp}
Following the same procedure as in Section \ref{subsec ci incomp}, we explain in the following how to construct infinitely many energy weak solutions emanating from a small neighborhood of a given subsolution data. 

\begin{theorem}\label{thm ci comp}
For any given $1 < \gamma \le 1 + \frac2n$, and $\IE, T > 0$, assume that $(\rho, V, \mathcal R, r)$ with $\rho > 0$ is a smooth $(\IE, T)$-energy compatible subsolution of the compressible Euler equations with initial value $(\rho^0, V^0, \mathcal R^0, r^0)$ such that $R(t,x) := \mathcal R(t,x) + r(t,x) \id > 0$ is positive definite for every $(t,x) \in \R_+ \times \T^n$. Then for any $\varepsilon > 0$, there exist infinitely many initial values $(\widetilde \rho^0_\varepsilon, \widetilde V^0_\varepsilon)$ such that $\widetilde \rho^0_\varepsilon > 0$ and 
\be\label{ci data cond1 comp}
\int_{\T^n} E(\widetilde \rho^0_\varepsilon, \widetilde V^0_\varepsilon) \,dx = \IE,  \qquad \|\widetilde \rho^0_\varepsilon - \rho^0\|^\gamma_{L^\gamma(\T^n)} + \LN \frac{\widetilde V^0_\varepsilon}{\sqrt{\widetilde\rho^0_\varepsilon}} - \frac{V^0}{\sqrt{\rho^0}} \RN^2_{L^2(\T^n)} < \varepsilon + \int_{\T^n} \textup{tr} R^0 \,dx,
\ee
where $R^0 := R|_{t = 0}$. For each of such initial values there exist infinitely many $(\widetilde \rho_\varepsilon, \widetilde V_\varepsilon) \in L^\infty(\R_+; L^\gamma(\T^n)) \times L^\infty(\R_+; L^{\frac{2\gamma}{\gamma+1}}(\T^n))$ which are global weak solutions to the compressible Euler equations with $(\widetilde \rho_\varepsilon, \widetilde V_\varepsilon) |_{t = 0} = (\widetilde \rho^0_\varepsilon, \widetilde V^0_\varepsilon)$ and 
\begin{equation}\label{ci data cond2 comp}
\int_{\T^n} E(\widetilde \rho_\varepsilon,\widetilde V_\varepsilon) \,dx \le \IE \quad \text{a.e. }\ t > 0.
\end{equation}
\end{theorem}
\begin{proof}
The proof follows the same idea as in Theorem \ref{thm ci incomp}. For the sake of completeness we provide the detailed argument. 

The smoothness of $(\rho, V, \mathcal R, r)$ and $\rho > 0$ implies that for any any $\varepsilon > 0$ there exists some small $t_0 < \frac{T}{2}$ such that $\forall \ t < t_0$,
\begin{equation}\label{cond ci 1 comp}
\begin{split}
& \sup_{0\le t \le t_0} \|\rho - \rho^0\|^\gamma_{L^\gamma(\T^n)} < \frac{\varepsilon}{2}, \\
& t_0 \LB \sup_{0\le t \le t_0} \LC \frac{\LN \partial_t p(\rho) \RN_{L^1(\T^n)}}{\gamma - 1} + 4 \LC \IE \RC^2\LN \frac{\partial_t \sqrt{\rho}}{\sqrt{\rho^0}} \RN_{L^\infty(\T^n)} \RC + 2 \IE \left\| \nabla \LC \frac{V^0}{\rho^0} \RC \right\|_{L^\infty(\T^n)} \RB  < \frac{\varepsilon}{4}, 
\end{split}
\end{equation}

The positivity of $R$ allows us to apply our convex integration program as in Proposition \ref{prop reduction} on $[0, t_0]$. However this could potentially lead to a loss of total energy resulting from the potential energy part. To resolve this issue, we will introduce the `compensating potential energy density' 
\begin{equation}\label{def comp PE}
r_c(t) := \LC \frac{2}{n(\gamma -1)} \RC \fint_{\T^n} r(t, x) \,dx.
\end{equation}
It is easy to see that $r_c(t)$ only depends on time, and it satisfies
\begin{equation}\label{prop comp PE}
\begin{split}
& \Div \LC r_c(t) \id \RC = 0, \\
& r_c(t) \ge 0, \\
\int_{\T^n} & \LB r(t,x) + r_c(t) \RB \,dx = \frac{2}{n(\gamma-1)} \int_{\T^n} r(t,x) \,dx.
\end{split}
\end{equation}
From the first and second properties above we can verify that $(\rho, V, \mathcal R, r + r_c)$ solves 
\begin{equation}\label{comp sub sys}
\left\{
\begin{split}&\rho_t+\Div V=0,
\\& V_t+\Div \left(\frac{V \otimes V}{\rho}+ \mathcal R+ (r + r_c) \id+p(\rho) \id\right)=0,
\end{split}\right.
\end{equation}
with $\widetilde R := \mathcal R+ (r + r_c) \id > 0$. 

Therefore we can perform convex integration for the above system \eqref{comp sub sys} to produce infinitely many weak Euler solutions $(\hat \rho, \hat V)$ on $[0, t_0]$ with
\[
\hat \rho = \rho, \quad \hat V |_{t = 0} = V^0, \quad \hat V|_{t = t_0} = V|_{t = t_0}.
\]
At the energy level, from \eqref{ci energy density} and the third property of \eqref{prop comp PE}, we know that for a.e. $t \in [0, t_0]$,
\be\label{e-eq comp}
\begin{split}
\int_{\T^n} E(\rho, \hat V) \,dx & = \int_{\T^n} \LC E(\rho, V) + \frac12 \trace \widetilde{R} \RC \,dx \\
& = \int_{\T^n} \LC E(\rho, V) + \frac12 \trace \mathcal R + \frac{n}{2} (r + r_c) \RC \,dx \\
& = \int_{\T^n} \LC E(\rho, V) + \frac12 \trace \mathcal R + \frac{r}{\gamma - 1} \RC \,dx = \IE,
\end{split}
\ee
where the last equality follows from \ref{es saturate comp}. 

Hence we can choose infinitely many $\widetilde t \in (0, t_0)$ to define the initial data
\[
(\widetilde \rho^0_\varepsilon, \widetilde V^0_\varepsilon) := (\rho, \hat V)|_{t = \widetilde t}.
\]
Therefore
\[
\int_{\T^n} E(\widetilde \rho^0_\varepsilon, \widetilde V^0_\varepsilon) \,dx \le \IE. 
\]
Meanwhile, 
\begin{align*}
\frac12\int_{\T^n} \LV \frac{\widetilde V^0_\varepsilon}{\sqrt{\widetilde\rho^0_\varepsilon}} - \frac{V^0}{\sqrt{\rho^0}} \RV^2 \,dx & = \int_{\T^n} \LB \frac{|\widetilde V^0_\varepsilon|^2}{2\widetilde \rho^0}  - \frac{|V^0|^2}{2\rho^0} +  \frac{|V^0|^2 - \widetilde V^0_\varepsilon \cdot V^0}{\rho^0} + \frac{\widetilde V^0_\varepsilon}{\sqrt{\widetilde \rho^0_\varepsilon}} \cdot \frac{V^0}{\sqrt{\rho^0}} \LC \sqrt{\frac{\widetilde \rho^0_\varepsilon}{\rho^0}} - 1 \RC \RB \,dx \\
& = \int_{\T^n} \frac{|\widetilde V^0_\varepsilon|^2}{2\widetilde \rho^0} \,dx - \LC \IE - \int_{\T^n} \LC \frac{p(\rho^0)}{\gamma - 1} + \frac12 \textup{\trace} \mathcal R^0 + \frac{r^0}{\gamma - 1}  \RC \,dx \RC \\
& \quad - \int^{\widetilde t}_0 \int_{\T^n} \partial_t \hat V \cdot \frac{V^0}{\rho^0} \,dxdt + \int_{\T^n} \frac{\widetilde V^0_\varepsilon}{\sqrt{\widetilde \rho^0_\varepsilon}} \cdot \frac{V^0}{\sqrt{\rho^0}} \LC \sqrt{\frac{\widetilde \rho^0_\varepsilon}{\rho^0}} - 1 \RC \,dx \\
& \le \IE - \int_{\T^n} \frac{p(\widetilde \rho^0_\varepsilon)}{\gamma - 1} \,dx - \LC \IE - \int_{\T^n} \LC \frac{p(\rho^0)}{\gamma - 1} + \frac12 \textup{\trace} \LC \mathcal R^0 + r^0 \id \RC \RC \,dx \RC \\
& \quad + \LV \int^{\widetilde t}_0 \int_{\T^n} \LC \frac{\hat V \otimes \hat V}{\rho} + p(\rho) \RC : \nabla \LC \frac{V^0}{\rho^0} \RC \,dxdt \RV + 4 \LC \IE \RC^2 \LN \sqrt{\frac{\widetilde \rho^0_\varepsilon}{\rho^0}} - 1 \RN_{L^\infty} \\
& \le \int_{\T^n} \frac{|p(\rho^0) - p(\widetilde \rho^0_\varepsilon)|}{\gamma - 1} \,dx + \frac12 \int_{\T^n} \textup{\trace} \LC \mathcal R^0 + r^0 \id \RC\,dx + 2 \widetilde t \IE  \left\| \nabla \LC \frac{V^0}{\rho^0} \RC \right\|_{L^\infty} \\
& \quad + 4 \widetilde t \LC \IE \RC^2 \sup_{0 \le t \le \widetilde t} \LN \frac{\partial_t \sqrt{\rho}}{\sqrt{\rho^0}} \RN_{L^\infty} \\
& < \frac12 \LC \frac{\varepsilon}{2} +  \int_{\T^n} \textup{\trace} R^0\,dx \RC \qquad \text{by \eqref{cond ci 1 comp}}.
\end{align*}
This together with \eqref{cond ci 1 comp} proves \eqref{ci data cond1 comp}. 

For each of the above initial data $(\widetilde \rho^0_\varepsilon, \widetilde V^0_\varepsilon)$ we define on $\R_+ \times \T^n$
\begin{equation*}
(\widetilde \rho_\varepsilon, \widetilde V_\varepsilon)(t, \cdot) = \left\{\begin{array}{ll}
(\rho, \hat V) (t + \widetilde t, \cdot), \quad & \ \text{ for } \ t \le t_0 - \widetilde t, \\
(\rho, \breve V) (t + \widetilde t, \cdot),  & \ \text{ for }\ t \ge t_0 - \widetilde t,
\end{array}\right.
\end{equation*}
where $(\rho, \breve V)$ is any weak solution to the compressible Euler equations on $[t_0, \infty)$ constructed from convex integrating the energy compatible subsolution $(\rho, V, \mathcal R, r)$ on $[t_0, \infty)$ (note that our convex integration scheme leaves $\rho$ unchanged). This way
\[
(\rho, \hat V)(t_0) = (\rho, V)(t_0) = (\rho, \breve V)(t_0).
\]
Therefore $(\widetilde \rho_\varepsilon, \widetilde V_\varepsilon)$ is indeed a weak solution to the compressible Euler equations on $\R_+ \times \T^n$, and by construction we know that $(\widetilde \rho_\varepsilon, \widetilde V_\varepsilon)|_{t = 0} = (\widetilde \rho^0_\varepsilon, \widetilde V^0_\varepsilon)$ and $(\widetilde \rho_\varepsilon, \widetilde V_\varepsilon)$ satisfies \eqref{ci data cond2 comp}.
\end{proof}

\subsection{Smooth energy compatible strict subsolutions}\label{subsec strict subsoln comp}

The next step is to find a way to construct from an energy compatible subsolution to an energy compatible strict subsolution in the sense that $R := \mathcal R + r \id > 0$. This can be achieve by a similar convex combination technique as in Section \ref{subsec strict subsoln}. The difference is that we will only apply the convex combination on the density variable with a nontrivial constant state.

We first state the following lemma which is a compressible version of Lemma \ref{lem conv combo}. The proof follows along a very similar argument as before, and hence we omit it. 
\begin{lemma}\label{lem conv combo comp}
Let $(\Omega, \mu)$ be a probability space, that is, $\mu$ is a nonnegative measure on $\Omega$ such that $\mu(\Omega) = 1$. Fix $\IE, T > 0$. Let $(\rho, V, \mathcal R, r)$ be a measurable function from $\R_+ \times \T^n \times \Omega$ to $\R^n \times \R^{n \times n}$ such that for a.e. $\omega \in \Omega$, $(\rho_\omega, V_\omega, \mathcal R_\omega, r_\omega) := \LC \rho(\cdot, \omega), V(\cdot, \omega), \mathcal R(\cdot, \omega), r(\cdot, \omega) \RC$ is an $(\IE, T)$-energy compatible subsolution to the compressible Euler equations. Denote
\be\label{conv combo comp}
\begin{split}
& (\overline \rho, \overline V) := \mathbb{E} (\rho_\omega, V_\omega), \\
& \overline{\mathcal R} := \mathbb E(\mathcal R_\omega) + \mathbb E\LC \frac{V_\omega \otimes V_\omega}{\rho_\omega}\RC - \frac{\overline V \otimes \overline V}{\overline \rho}, \qquad \overline r := \mathbb{E}(r_\omega) + \mathbb{E}\LC p(\rho_\omega) \RC - p(\overline \rho),
\end{split}
\ee
where 
\[
\mathbb E(f_\omega) := \int_\Omega f_\omega d\mu(\omega).
\]
Then $(\overline \rho, \overline V, \overline{\mathcal R}, \overline r)$ is also an $(\IE, T)$-energy compatible subsolution.
\end{lemma}

With the above we are ready to state the main result of this subsection. 
\begin{theorem}\label{thm smoothing and conv combo comp}
Given $\IE, T > 0$ 
let $(\rho, V, \mathcal R, r)$ be an $(\IE, T)$-energy compatible subsolution to the compressible Euler equations with initial data $(\rho^0, V^0, \mathcal R^0, r^0)$ satisfying $\rho^0 \not\equiv 0$ and 
\begin{equation}\label{ic comp reg}
\begin{split}
& V^0 \not\equiv 0, \qquad \text{or} \qquad V^0 \equiv 0\  \text{ but }\ \rho^0 \text{ is not a constant,}\\
& \int_{\T^n} E(\rho^0, V^0) \,dx = \IE \quad \text{and} \quad \int_{\T^n} \LC \textup{\trace} \mathcal R^0 + \frac{r^0}{\gamma - 1} \RC \,dx = 0.
\end{split}
\end{equation}
Then for any $\varepsilon > 0$ there exists a smooth $(\IE, \frac{T}{2})$-energy compatible subsolution $(\widetilde \rho, \widetilde V, \widetilde{\mathcal R}, \widetilde r)$ with 
\be\label{strict subsoln comp}
\begin{split}
& \widetilde \rho > 0, \qquad \widetilde{\mathcal R} + \widetilde r \id > 0, \\
& \LN \frac{\widetilde V^0}{\sqrt{\widetilde\rho^0}} - \frac{V^0}{\sqrt{\rho^0}} \RN^2_{L^2} + \LN \widetilde \rho^0 - \rho^0 \RN^\gamma_{L^\gamma} + \int_{\T^n} \LC \textup{\trace} \widetilde{\mathcal R}^0 + \frac{\widetilde r^0}{\gamma -1} \RC \,dx < \varepsilon.
\end{split}
\ee 
\end{theorem}
\begin{proof}
Similarly as in the proof of Theorem \ref{thm smoothing and conv combo}, we introduce a mollifier $\varphi \in C^\infty(\R \times \T^n)$ with $\varphi > 0$ and $\int \varphi = 1$. For $\alpha > 0$ we define the scaled mollifier
\[
\varphi_\alpha(t, x) := \frac{1}{\alpha^{n+1}} \varphi\LC \frac{t}{\alpha}, \frac{x}{\alpha} \RC.
\]
We will choose
\be\label{prob space}
\Omega := \{\omega := (s,y)\in\R_+\times \mathbb T^n\}, \qquad d\mu(s,y) := \varphi_\alpha(s,y)\,ds\,dz,
\ee
and for $s < T/2$, define
\[
(\rho_\omega, V_\omega, \mathcal R_\omega, r_\omega)(\cdot, \cdot, \omega) := (\rho, V, \mathcal R, r)(\cdot + s, \cdot + y)
\]
on $[0, T/2] \times \T^n$. Thus $(\rho_\omega, V_\omega, \mathcal R_\omega, r_\omega)(\cdot, \cdot, \omega)$ is an $(\IE, \frac{T}{2})$-energy compatible subsolution for the compressible Euler equations. 

From the definition of $\Omega$ and $\mu$ we see that
\[
\mathbb{E}(f_\omega) = \varphi_\alpha \ast f.
\]
Applying Lemma \ref{lem conv combo comp} we obtain another $(\IE, \frac{T}{2})$-energy compatible subsolution $(\underline \rho, \underline V, \underline{\mathcal R}, \underline r)$, where
\begin{equation}\label{mollify}
\begin{split}
& (\underline \rho, \underline V) := \varphi_\alpha \ast (\rho, V), \\
& \underline{\mathcal R} := \varphi_\alpha \ast \LC \mathcal R +  \frac{V \otimes V}{\rho}\RC - \frac{\underline V \otimes \underline V}{\underline \rho}, \\ 
& \underline r := \varphi_\alpha \ast \LC r +  p(\rho) \RC - p(\underline \rho).
\end{split}
\end{equation}
We further know that $(\underline \rho, \underline V, \underline{\mathcal R}, \underline r)$ is a smooth energy compatible subsolution. The choice of $\varphi$ ensures that $\underline \rho > 0$. For any given $\varepsilon > 0$, by taking $\alpha$ sufficiently small we have
\be\label{ic close comp}
\LN \frac{\underline V^0}{\sqrt{\underline\rho^0}} - \frac{V^0}{\sqrt{\rho^0}} \RN^2_{L^2} + \LN \underline \rho^0 - \rho^0 \RN^\gamma_{L^\gamma} + \int_{\T^n} \LC \trace \underline{\mathcal R}^0 + \frac{\underline r^0}{\gamma -1} \RC \,dx < \frac{\varepsilon}{4}.
\ee

Now we can apply the convex combination method. As in the proof of Theorem \ref{thm smoothing and conv combo}, we will work with an atomic measure. For $\lambda \in (0, 1)$, set 
\[
\Omega := \{1, 2\}, \qquad \mu := \lambda \delta_{\omega = 1} + (1-\lambda) \delta_{\omega=2},
\]
and consider two $(\IE, \frac{T}{2})$-energy compatible subsolutions
\[
(\rho_1, V_1, \mathcal R_1, r_1) = (\hat \rho, 0, 0, 0), \quad (\rho_2, V_2, \mathcal R_2, r_2) =  (\underline \rho, \underline V, \underline{\mathcal R}, \underline r),
\]
where $\hat \rho$ is a constant such that
\[
\frac{p(\hat \rho)}{\gamma -1} = \frac{\IE}{|\T^n|}.
\]
Applying Lemma \ref{lem conv combo comp} again yields a smooth $(\IE, \frac{T}{2})$-energy compatible subsolution $(\widetilde \rho_1, \widetilde V_1, \widetilde{\mathcal R}_1, \widetilde r_1)$ with
\begin{equation*}
\begin{split}
& (\widetilde \rho_1, \widetilde V_1) = \LC \lambda \hat \rho + (1 - \lambda) \underline \rho, \lambda \underline V \RC, \\
& \widetilde{\mathcal R}_1 = (1 - \lambda) \underline{\mathcal R} + (1-\lambda) \LC \frac{\underline V \otimes \underline V}{\underline \rho} \RC - \frac{(1 - \lambda)^2 \underline V \otimes \underline V}{\lambda \hat \rho + (1-\lambda) \underline\rho}, \\
& \widetilde r_1 = (1 - \lambda) \underline r + \lambda p(\hat \rho) + (1 - \lambda) p(\underline \rho) - p\big( \lambda \hat \rho + (1-\lambda) \underline \rho \big) =: (1-\lambda) \underline r + \hat r.
\end{split}
\end{equation*}
From \eqref{ic close comp} and continuity it follows that for $\lambda$ sufficiently small
\be\label{ic close cc comp}
\LN \frac{\widetilde V^0_1}{\sqrt{\widetilde\rho^0_1}} - \frac{V^0}{\sqrt{\rho^0}} \RN^2_{L^2} + \LN \widetilde \rho^0_1 - \rho^0 \RN^\gamma_{L^\gamma} + \int_{\T^n} \LC \trace \widetilde{\mathcal R}_1^0 + \frac{\widetilde r_1^0}{\gamma -1} \RC \,dx < \frac{\varepsilon}{2}.
\ee

Strict convexity of $p$ indicates that 
\[
\widetilde r_1 \equiv 0 \quad \Longleftrightarrow \quad \underline r = \hat r \equiv 0. 
\] 
We claim that
\be\label{r nonzero}
\widetilde r_1 \not \equiv 0.
\ee
If $\underline r \equiv 0$, then from \eqref{mollify} we see that and convexity of $p$
\[
\underline r \equiv 0, \qquad \varphi_\alpha \ast p(\rho) - p(\varphi_\alpha\ast \rho) \equiv 0.
\]
The second identity yields that $\rho$ is a constant, specifically,
\begin{equation*}
\rho = \fint_{\T^n} \rho^0 \,dx.
\end{equation*}
When $\hat r \equiv 0$, it again follows from the strict convexity of $p$ that 
\[
\rho = \hat \rho. 
\]
Comparing the above two conditions, using the definition of $\hat \rho$ and applying Jensen's inequality we see that 
\be\label{jensen}
\fint_{\T^n} p(\rho^0)\,dx \ge p\LC \fint_{\T^n} \rho^0 \,dx \RC = p(\rho) = p(\hat \rho) = \frac{(\gamma - 1) \IE}{|\T^n|}.
\ee
On the other hand, \eqref{ic comp reg} implies that
\[
\fint_{\T^n} p(\rho^0)\,dx \le \frac{(\gamma - 1) \IE}{|\T^n|}.
\]
If $V^0 \not\equiv 0$ then the above must be a strict inequality, which contradicts \eqref{jensen}. 
Therefore $V^0 \equiv 0$ and equality in \eqref{jensen} must hold. Hence either $\rho^0$ is a constant or $p$ is linear on $\rho^0(\T^n)$. The explicit form of $p$ suggests the former, but this contradicts \eqref{ic comp reg}. Therefore \eqref{r nonzero} holds. 

Since we are using a nonvanishing mollifier, it is easy to see that 
\[
\varphi_\alpha \ast \widetilde r_1 > 0.
\]
Therefore applying mollification to $(\widetilde \rho_1, \widetilde V_1, \widetilde{\mathcal R}_1, \widetilde r_1)$ yields the desired smooth $(\IE, \frac{T}{2})$-energy compatible subsolution $(\widetilde \rho, \widetilde V, \widetilde{\mathcal R}, \widetilde r)$ with $\widetilde \rho > 0$. From \eqref{mollify} we know that they take the form
\begin{equation*}
\begin{split}
& (\widetilde \rho, \widetilde V) := \varphi_\alpha \ast (\widetilde \rho_1, \widetilde V_1), \\
& \widetilde{\mathcal R} := \varphi_\alpha \ast \LC \widetilde{\mathcal R}_1 +  \frac{\widetilde V_1 \otimes \widetilde V_1}{\widetilde \rho_1}\RC - \frac{\widetilde V \otimes \widetilde V}{\widetilde \rho}, \\ 
& \widetilde r := \varphi_\alpha \ast \LC \widetilde r_1 +  p(\widetilde \rho_1) \RC - p(\widetilde \rho) \ge \varphi_\alpha \ast \widetilde r_1 > 0. 
\end{split}
\end{equation*}
By taking $\alpha$ small enough and using \eqref{ic close comp} and \eqref{ic close cc comp} we prove \eqref{strict subsoln comp}.
\end{proof}

\subsection{Density of wild data for compressible Euler equations}\label{subsec density comp}

We can now prove Theorem \ref{thm main comp}.
\begin{proof}
The strategy is the same as in the incompressible case. We first regularize the data to $(\varrho^0_\varepsilon, U^0_\varepsilon) \in C^\infty(\T^n)$ with the property that
\be\label{comp ic reg}
\| \varrho^0_\varepsilon - \varrho^0 \|_{L^1(\T^n)} + \|U^0_\varepsilon - U^0\|_{L^1(\T^n)} < \varepsilon. 
\ee
From Jensen's inequality we have
\[
\int_{\T^n} E(\varrho^0_\varepsilon, U^0_\varepsilon) \,dx \le \int_{\T^n} E(\varrho^0, U^0) \,dx.
\]
Hence the regularized data $(\varrho^0_\varepsilon, U^0_\varepsilon)$ satisfies \eqref{comp ic}. In particular if we use a non-vanishing mollification kernel then $\varrho^0_\varepsilon > 0$.

On the other hand, defining $\Omega_\delta := \{x\in \T^n: \ \varrho^0(x) > \delta \}$ and $\Omega^0 := \{ x\in \T^n:\ \varrho^0(x) = 0 \}$ we see that for a fixed $\delta > 0$, as $\varepsilon \to 0$
\[
E(\varrho^0_\varepsilon, U^0_\varepsilon) \longrightarrow E(\varrho^0, U^0) \quad \text{a.e. }\ \Omega_\delta.
\]
Hence from Fatou's lemma we have
\begin{equation*}
\begin{split}
\limsup_{\varepsilon \to 0} \int_{\Omega_\delta \cup \Omega^0} E(\varrho^0_\varepsilon, U^0_\varepsilon) \,dx \ge \int_{\Omega_\delta \cup \Omega^0} E(\varrho^0, U^0) \,dx \to \int_{\T^n} E(\varrho^0, U^0) \,dx \quad \text{as} \quad \delta \to 0.
\end{split}
\end{equation*}
Therefore for $\varepsilon$ sufficiently small 
\be\label{comp ic energy}
\LN E(\varrho^0_\varepsilon, U^0_\varepsilon) - E(\varrho^0, U^0) \RN_{L^1(\T^n)} = \int_{\T^n}E(\varrho^0, U^0) \,dx - \int_{\T^n} E(\varrho^0_\varepsilon, U^0_\varepsilon) \,dx < \varepsilon.
\ee

From \eqref{comp ic reg} and \eqref{comp ic energy}, and further refining the mollification scale if necessary, we have 
\be\label{comp ic 1}
\|\varrho^0_\varepsilon - \varrho^0\|_{L^\gamma}^\gamma + \LN \frac{U^0_\varepsilon}{\sqrt{\varrho^0_\varepsilon}} - \frac{U^0}{\sqrt{\varrho^0}} \RN^2_{L^2}  < \frac{\varepsilon}{9}.
\ee

Denote
\[
\IE_\varepsilon := \int_{\T^n} E(\varrho^0_\varepsilon, U^0_\varepsilon) \,dx.
\]
We can apply Proposition \ref{prop e-subsoln comp} to find a $T>0$ and an $(\IE_\varepsilon, T)$-energy compatible subsolution $(\varrho_\varepsilon, U_\varepsilon, \mathcal R_\varepsilon, r_\varepsilon)$ with initial data $(\varrho_\varepsilon, U_\varepsilon)|_{t = 0} = (\varrho^0_\varepsilon, U^0_\varepsilon)$ and satisfying
\[
\mathcal R_\varepsilon \equiv 0, \quad r_\varepsilon \equiv 0 \quad \text{for}\quad t\in [0, T].
\]
Then we use Theorem \ref{thm smoothing and conv combo comp} to produce from $(\varrho_\varepsilon, U_\varepsilon, \mathcal R_\varepsilon, r_\varepsilon)$ a smooth $(\IE_\varepsilon, \frac{T}{2})$-energy compatible subsolution $(\widetilde \varrho_\varepsilon, \widetilde U_\varepsilon, \widetilde{\mathcal R}_\varepsilon, \widetilde r_\varepsilon)$ with initial data $(\widetilde \varrho^0_\varepsilon, \widetilde U^0_\varepsilon, \widetilde{\mathcal R}^0_\varepsilon, \widetilde r^0_\varepsilon)$ and satisfying (from \eqref{strict subsoln comp})
\be\label{comp ic 2}
\|\widetilde\varrho^0_\varepsilon - \varrho^0_\varepsilon\|_{L^\gamma}^\gamma + \LN \frac{\widetilde U^0_\varepsilon}{\sqrt{\widetilde \varrho^0_\varepsilon}} - \frac{U^0_\varepsilon}{\sqrt{\varrho^0_\varepsilon}} \RN^2_{L^2} + \int_{\T^n} \LC \textup{\trace} \widetilde{\mathcal R}^0_\varepsilon + \frac{\widetilde r^0_\varepsilon}{\gamma -1} \RC \,dx < \frac{\varepsilon}{9}.
\ee

Using the positivity of $\widetilde R_\varepsilon := \widetilde{\mathcal R}_\varepsilon + r_\varepsilon \id$ we may employ Theorem \ref{thm ci comp} to convex integrate. This way we obtain infinitely many initial data $(\rho^0, V^0)  \in L^\gamma(\T^n) \times L^{\frac{2\gamma}{\gamma - 1}}(\T^n)$ satisfying (from \eqref{ci data cond1 comp})
\be\label{comp ic 3}
 \|\rho^0 - \widetilde\varrho^0_\varepsilon\|_{L^\gamma}^\gamma + \LN \frac{V^0}{\sqrt{\rho^0}} - \frac{\widetilde U^0_\varepsilon}{\sqrt{\widetilde \varrho^0_\varepsilon}} \RN^2_{L^2} < \frac{\varepsilon}{9} + \int_{\T^n} \textup{tr} \widetilde R^0_\varepsilon \,dx,
\ee
each of which induces infinitely many weak solutions nfinitely many $(\rho, V) \in L^\infty(\R_+; L^\gamma(\T^n)) \times L^\infty(\R_+; L^{\frac{2\gamma}{\gamma+1}}(\T^n))$ to the compressible Euler equations such that 
\[
\int_{\T^n} E(\rho,V) \,dx \le \IE_\varepsilon = \int_{\T^n} E(\rho^0,V^0) \,dx  \quad \text{a.e. }\ t > 0.
\] 
Moreover for the initial data we have
\begin{align*}
\|\rho^0 - \varrho^0\|_{L^\gamma(\T^n)}^\gamma & + \LN \frac{V^0}{\sqrt{\rho^0}} - \frac{U^0}{\sqrt{\varrho^0}} \RN_{L^2(\T^n)}^2 \\
& \le 3\LC \|\rho^0 - \widetilde\varrho^0_\varepsilon\|_{L^\gamma}^\gamma + \|\widetilde\varrho^0_\varepsilon - \varrho^0_\varepsilon\|_{L^\gamma}^\gamma + \|\varrho^0_\varepsilon - \varrho^0\|_{L^\gamma}^\gamma \RC \\
& \quad + 3 \LC \LN \frac{V^0}{\sqrt{\rho^0}} - \frac{\widetilde U^0_\varepsilon}{\sqrt{\widetilde \varrho^0_\varepsilon}} \RN^2_{L^2} + \LN \frac{\widetilde U^0_\varepsilon}{\sqrt{\widetilde \varrho^0_\varepsilon}} - \frac{U^0_\varepsilon}{\sqrt{\varrho^0_\varepsilon}} \RN^2_{L^2} + \LN \frac{U^0_\varepsilon}{\sqrt{\varrho^0_\varepsilon}} - \frac{U^0}{\sqrt{\varrho^0}} \RN^2_{L^2} \RC \\
\text{by \eqref{comp ic 3} and \eqref{comp ic 1}} \quad & < 3 \LC \frac{\varepsilon}{9} + \int_{\T^n} \textup{tr} \widetilde R^0_\varepsilon \,dx + \|\widetilde\varrho^0_\varepsilon - \varrho^0_\varepsilon\|_{L^\gamma}^\gamma + \LN \frac{\widetilde U^0_\varepsilon}{\sqrt{\widetilde \varrho^0_\varepsilon}} - \frac{U^0_\varepsilon}{\sqrt{\varrho^0_\varepsilon}} \RN^2_{L^2} + \frac{\varepsilon}{9} \RC < \varepsilon,
\end{align*} 
where in the last inequality we used \eqref{comp ic 2} and the fact that $\gamma \le 1 + \frac{2}{n}$. Therefore we obtain \eqref{comp ic}, and hence complete the proof of the theorem. 
\end{proof}

\addtocontents{toc}{\protect\setcounter{tocdepth}{0}}
\section*{Acknowledgement} 
Robin Ming Chen is partially supported by the NSF grant: DMS 1907584. Alexis Vasseur is partially supported by the NSF grant: DMS 1907981.   Cheng Yu is is partially supported by the Collaboration Grants for Mathematicians from Simons Foundation.

\addtocontents{toc}{\setcounter{tocdepth}{1}} 
\bibliographystyle{siam}
\bibliography{Reference}

\end{document}